\documentclass[11pt]{article}\textwidth 160mm\textheight 235mm
\usepackage{amsfonts} 
\usepackage{amssymb} 
\usepackage{graphicx}
\usepackage{amsmath}
\usepackage{amsthm} 
\usepackage{mathrsfs}
\usepackage{tikz} 
\usepackage{rotating,tabularx}
\usepackage{cancel} \usepackage{todonotes}
\usetikzlibrary{matrix} \usetikzlibrary{arrows,shapes}
\usepackage{cite} \usepackage{hyperref}
\hypersetup{colorlinks=true,urlcolor=blue}
\expandafter\let\expandafter\oldproof\csname\string\proof\endcsname
\let\oldendproof\endproof \renewenvironment{proof}[1][\proofname]{
\oldproof[\ttfamily\scshape \bf #1.] }{\oldendproof}
  
\def\emp{\emptyset}  
\def\dom{{\rm dom}\,}  \def\epi{{\rm epi\,}}
  \def\min{\mbox{\rm minimize}}
 
   \def\B{\mathbb
B}   
\def\ox{\overline{x}} \def\oy{\overline{y}} \def\oz{\overline{z}}
  
\def\tto{\rightrightarrows} 
\def\Tilde{\widetilde} \def\Bar{\overline} \def\ra{\rangle}
\def\la{\langle}  \def\epsilon{\varepsilon}
\def\ox{\bar{x}} \def\oy{\bar{y}} \def\oz{\bar{z}}
 \def\ov{\bar{v}}

\def\gph{\mbox{\rm gph}\,} \def\epi{\mbox{\rm epi}\,}
 \def\dom{\mbox{\rm dom}\,}

 \def\dn{\downarrow} \def\O{\Omega}
\def\ph{\varphi} \def\emp{\emptyset} \def\st{\stackrel}
\def\oR{\Bar{\R}}   
\def\al{\alpha}  \def\N{{\rm I\!N}}
\def\R{\mathbb{R}}

\begin{document}
\newtheorem{Assumption}{Assumption}
\newtheorem{Theorem}{Theorem}[section]
\newtheorem{Proposition}[Theorem]{Proposition}
\newtheorem{Remark}[Theorem]{Remark}
\newtheorem{Lemma}[Theorem]{Lemma}
\newtheorem{Corollary}[Theorem]{Corollary}
\newtheorem{Definition}[Theorem]{Definition}
\newtheorem{Example}[Theorem]{Example}
\newtheorem{Algorithm}[Theorem]{Algorithm}
\newtheorem{Problem}[Theorem]{Problem}
\renewcommand{\theequation}{{\thesection}.\arabic{equation}}
\renewcommand{\thefootnote}{\fnsymbol{footnote}} \begin{center}
{\bf\Large  {Generalized Damped Newton Algorithms in Nonsmooth Optimization via Second-Order Subdifferentials}}\\[1ex]
Pham Duy Khanh\footnote{Department of Mathematics, HCMC University
of Education, Ho Chi Minh City, Vietnam.}, Boris S.
Mordukhovich\footnote{Department of Mathematics, Wayne State
University, Detroit, Michigan, USA. E-mail: boris@math.wayne.edu. Research of this author was partly supported by the US National Science Foundation under grants DMS-1512846 and DMS-1808978, by the US Air Force Office of Scientific Research under grant \#15RT0462, and by the Australian Research Council under Discovery Project DP-190100555.},
Vo Thanh Phat \footnote{Department of Mathematics, HCMC
University of Education, Ho Chi Minh City, Vietnam and Department of
Mathematics, Wayne State University, Detroit, Michigan, USA.
E-mails: phatvt@hcmue.edu.vn; phatvt@wayne.edu. Research of this author was partly supported by the US National Science Foundation under grants DMS-1512846 and DMS-1808978, and by the US Air Force Office of Scientific Research under grant \#15RT0462.}, Dat Ba Tran \footnote{Department of
Mathematics, Wayne State University, Detroit, Michigan, USA.
E-mails:  tranbadat@wayne.edu. Research of this author was partly supported by the US National Science Foundation under grant DMS-1808978.}.

 \end{center}
\small{\bf Abstract}.  The paper proposes and develops new globally convergent algorithms of the generalized damped Newton type for solving important classes of nonsmooth optimization problems. These algorithms are based on the theory and calculations of second-order subdifferentials of nonsmooth functions with employing the machinery of second-order variational analysis and generalized differentiation. First we develop a globally superlinearly convergent damped Newton-type algorithm for the class of continuously differentiable functions with Lipschitzian gradients, which are nonsmooth of second order. Then we design such a globally convergent algorithm to solve a  structured class of nonsmooth quadratic composite problems with extended-real-valued cost functions, which typically arise in machine learning and statistics. Finally, we present the results of numerical experiments and compare the performance of our main algorithm applied to an important class of Lasso problems with those achieved by other first-order and second-order optimization algorithms.\\[1ex]
{\bf Keywords}. Variational analysis and nonsmooth optimization, damped Newton methods, global convergence, tilt stability of minimizers, superlinear convergence, Lasso problems.\vspace*{-0.2in}

\normalsize\section{Introduction}\label{intro}

This paper is mainly devoted to the design, justification, and applications of {\em globally convergent} Newton-type algorithms to   solve  nonsmooth (of the first or second order) optimization problems  in finite-dimensional spaces. Considering the unconstrained optimization problem
\begin{equation}\label{opproblem}
\min\;\varphi(x)\;\text{ subject to }\;x\in\R^n
\end{equation}
with a continuously differentiable (${\cal C}^1$-smooth) cost function $\ph\colon\R^n\to\R$, recall that one of the most natural and efficient approaches to solve \eqref{opproblem} globally is by using {\em line search methods}; see, e.g., \cite{JPang,Solo14,nw} and the references therein. Given a starting point $x^0\in\R^n$, such methods construct an iterative procedure of the form
\begin{equation}\label{linealgorithm}
x^{k+1}:=x^k+\tau_k d^k\quad\text{for all }\;k\in\N:=\{1,2,\ldots\},
\end{equation}
where $\tau_k\ge 0$ is a {\em step size} at iteration $k$, and where $d^k\ne 0$ is a {\em search direction}. The precise choice of $d^k$ and $\tau_k$ at each iteration in \eqref{linealgorithm} distinguishes one algorithm from another. The main goal of line search methods is to construct a sequence of iterates $\{x^k\}$ such that the corresponding sequence $\{\varphi(x^k)\}$ is decreasing. Recall also that the condition $\langle\nabla\varphi(x^k),d^k\rangle<0$ on $d^k$ ensures that it is a {\em descent direction} at $x^k$, i.e., there exists $\bar{\tau}_k\in(0,1]$ such that $\varphi(x^k+\tau d^k)<\varphi(x^k)$ for all $\tau\in[0,\bar{\tau}_k]$. There are many choices of the direction $d^k$ that satisfies this condition. For instance, a classical choice for the search direction is $d^k:=-\nabla\varphi(x^k)$ when the resulting algorithm is known as the {\em gradient algorithm} or {\em steepest descent method}; see \cite{BeckAl,Boyd,JPang,Solo14,nest,pol} for more details and impressive further developments of gradient and subgradient methods.

If $\varphi$ is twice continuously differentiable (${\cal C}^2$-smooth) and the Hessian matrix $\nabla^2\varphi(x^k)$ is positive-definite for each $k\in\N$, then another choice of search directions in \eqref{linealgorithm} is provided by solving the linear equation
\begin{equation}\label{Newtonequa}
-\nabla\varphi(x^k)=\nabla^2\varphi(x^k)d^k,
\end{equation}
where $d^k$ is known as a {\em Newton direction}. In this case, algorithm \eqref{linealgorithm} with the {\em backtracking line search} is called the {\em damped/guarded Newton method} \cite{BeckAl,Boyd} to distinguish it from the {\em pure Newton method}, which uses a fixed step size $\tau=1$; see, e.g., the books \cite{Donchev09,JPang,Solo14,Klatte} with the comprehensive commentaries and references therein. It has been well recognized that the latter method exhibits a {\em local} convergence with {\em quadratic} rate.

There exist various extensions of the pure Newton method to solve unconstrained optimization problems \eqref{opproblem}, where the cost functions $\ph$ are not ${\cal C}^2$-smooth but belong merely to the class ${\cal C}^{1,1}$ of continuously differentiable functions with Lipschitz continuous gradients, i.e., being nonsmooth of second order. We refer the reader to
\cite{Bonnans,Donchev09,JPang,Solo14,BorisKhanhPhat,Klatte,BorisEbrahim,qs,Ul} and the vast bibliographies therein for a variety of results in this direction, where mostly {\em a local superlinear} convergence rate was achieved, while in some publications certain globalization procedures were also suggested and investigated.\vspace*{0.03in}

The first goal of this paper is to develop a {\em globally convergent damped Newton method} of type \eqref{linealgorithm}, \eqref{Newtonequa} to solve problems \eqref{opproblem} with cost functions $\ph$ of class ${\cal C}^{1,1}$. Our approach is based on replacing the classical Hessian matrix $\nabla^2\varphi$ in equation \eqref{Newtonequa} by the inclusion
\begin{equation}\label{geneNewton}
-\nabla\varphi(x^k)\in\partial^2\varphi(x^k)(d^k),\quad k=0,1,\ldots,
\end{equation}
where $\partial^2\varphi$ stands for {\em second-order subdifferential/generalized Hessian} of $\varphi$ in the sense of Mordukhovich \cite{m92}. This construction has been largely used in variational analysis and its applications with deriving comprehensive calculus rules and complete computations of $\partial^2\ph$ for broad classes of composite functions that often appeared in important problems of optimization, optimal control, stability, applied sciences, etc.; see, e.g., \cite{chhm,dsy,dr,dl,dmn,hmn,Mordukhovich06,Mor18,MorduNghia,BorisOutrata,mr,os,Poli,roc,yy} with further references therein. The second-order subdifferentials have been recently employed in \cite{BorisEbrahim} and \cite{BorisKhanhPhat} for the design and justifications of generalized algorithms of the {\em pure Newton type} to find {\em stable} local minimizers of \eqref{opproblem} as well as solutions of gradient equations and subgradient inclusions associated with ${\cal C}^{1,1}$ and prox-regular functions, respectively.

In this paper we obtain efficient conditions ensuring that the iterative sequence generated by the damped Newton-type algorithm in \eqref{linealgorithm}, \eqref{Newtonequa} is {\em well-defined} (i.e., the algorithm {\em solvability}), and that the iterative sequence {\em global converges} to a {\em tilt-stable} local minimizer of \eqref{opproblem} in the sense of Poliquin and Rockafellar \cite{Poli}. It is shown that the rate of convergence of our algorithm is at least {\em linear}, while the {\em superlinear} convergence of the algorithm is achieved under the additional {\em semismooth$^*$} assumption on $\nabla\varphi$ in the sense of Gfrerer and Outrata \cite{Helmut}.\vspace*{0.03in}

The next major goal of the paper is to design, for the first time in the literature, a globally convergent damped Newton algorithm of solving nonsmooth problems of {\em convex composite optimization} given in the form:
\begin{equation}\label{comp}
\min\;\varphi(x):=f(x)+g(x)\;\text{ subject to }\;x\in\R^n,
\end{equation}
where $f$ is a convex quadratic function defined by $f(x):=\frac{1}{2}\langle Ax,x\rangle+\langle b,x\rangle+\alpha$ with  $b\in\R^n$, $\alpha\in\R$, and $A\in\R^{n\times n}$ being a positive-semidefinite matrix, and where $g\colon\R^n\to\oR:=(-\infty,\infty]$ is a lower semicontinuous (l.s.c.) extended-real-valued convex function. Problems in this format frequently arise in many applied areas such as machine learning, compressed sensing, and image processing. Since $g$ is generally extended-real-valued, the unconstrained format \eqref{comp} encompasses problems of {\em constrained optimization}. If, in particular, $g$ is the indicator function of a closed and convex set, then \eqref{comp} becomes a constrained quadratic optimization problems studied, e.g., in the book \cite{nw} with numerous applications.  Problems of this type are important in their own sake, while they also appear as {\em subproblems} in various numerical algorithms including sequential quadratic programming (SQP) methods, augmented Lagrangian methods, proximal Newton methods, etc. One of the most well-known algorithms to solve \eqref{comp} is the forward-backward splitting (FBS) or proximal splitting method \cite{cp,lm}. Since this method is of first order, its rate of convergence is at most linear. Another approach to solve \eqref{comp} is to use second-order methods such as proximal Newton methods, proximal quasi-Newton methods, etc.; see, e.g., \cite{bf,lss,myzz}. Although the latter approach has several benefits over first-order methods (as rapid convergence and high accuracy), a severe limitation of these methods is the cost of solving subproblems.\vspace*{0.03in}

In this paper we offer a different approach to solve problems \eqref{comp} globally by developing a generalized damped Newton algorithm based on the second-order subdifferential scheme \eqref{geneNewton}, advanced machinery of second-order variational analysis with the usage of the {\em proximal mapping} for $g$. As revealed below, the latter mapping can be constructively computed for many particular classes of problems arising in machine learning, statistics, etc. Proceeding in this way, we justify the well-posedness and global linear convergence of the proposed algorithm for \eqref{comp} with presenting efficient conditions for its superlinear convergence.\vspace*{0.03in}

The last topic of this paper concerns applications of the our {\em generalized damped Newton method} (GDNM) to solving an important class of {\em Lasso problems}, which appear in many areas of applied sciences and are discussed in detail in what follows. Problems of this class can be written in form \eqref{comp} with a quadratic loss function $f$ and a nonsmooth regularizer function $g$ given in special norm-type forms. For such problems, all the parameters of GDNM and its justification (first- and second-order subdifferentials, proximal mappings, conditions for convergence and convergence rates) can be computed and expressed entirely in terms of the problem data, which thus leads us the constructive globally superlinearly convergent realization of GDNM. Finally, we conduct MATLAB numerical experiments of solving the basic version of the Lasso problem described by Tibshirani \cite{Tibshirani} and then compare the obtained numerical results with those achieved by using well-recognized first-order and second-order methods. They include: Alternating Direction Methods of Multipliers (ADMM) \cite{gabay,glomar}, Nesterov's Accelerated Proximal Gradient with Backtracking (APG) \cite{Nesterov,nest}, Fast Iterative Shrinkage-Thresholding Algorithm with constant step size (FISTA) \cite{BeckTebou}, and a highly efficient Semismooth Newton Augmented Lagrangian Method (SSNAL) from \cite{lsk}.\vspace*{0.05in}

The rest of the paper is organized as follows. Section~\ref{sec:pre} presents and discusses some basic notions of variational analysis and generalized differentiation, which are broadly used in the formulations and proofs of the main results. Section~\ref{sec:dampedC11} is devoted to the development and justification of the globally convergent GDNM to solve unconstrained optimization problems \eqref{opproblem} with ${\cal C}^{1,1}$ cost functions. In Section~\ref{sec:rate} we present results on the linear and superlinear convergence of GDNM for problems of ${\cal C}^{1,1}$ optimization. Section~\ref{sec:dampnon} addresses developing GDNM for nonsmooth problems of convex composite optimization with cost functions given as sums of convex quadratic and convex extended-real-valued ones.  {In Section~\ref{Lassosec} we specify the obtained results for the basic class of Lasso problems under consideration and present the results of numerical experiments and their comparison with other first-order and second-order methods for solving Lasso problems}. The concluding Section~\ref{sec:conclusion} summarizes the major contributions of the paper and discusses topics of future research.\vspace*{-0.05in}

\section{Preliminaries from Variational Analysis}\label{sec:pre}

In this section we review the needed background from variational analysis and generalized differentiation by following the books \cite{Mordukhovich06,Mor18,Rockafellar98}, where the reader can find more details and references. Our notation is standard in variational analysis and optimization and can be found in the aforementioned books.

Given a set $\Omega\subset\R^s$ with $\oz\in\O$, the (Fr\'echet)
{\em regular normal cone} to $\Omega$ at $\bar{z}\in\Omega$ is \begin{equation*}
\widehat{N}_\Omega(\bar{z}):=\Big\{v\in\R^s\;\Big|\;\limsup_{z\overset{\Omega}{\rightarrow}\bar{z}}\frac{\langle
v,z-\bar{z}\rangle}{\|z-\bar{z}\|}\le 0\Big\}, \end{equation*} where the symbol
$z\overset{\Omega}{\rightarrow}\bar{z}$ stands for $z\to\bar{z}$
with $z\in\Omega$. The (Mordukhovich) {\em limiting normal cone} to
$\Omega$ at $\bar{z}\in\Omega$ is defined by \begin{equation}\label{lnc}
N_\Omega(\bar{z}):=\big\{v\in\R^s\;\big|\;\exists\,z_k\st{\O}{\to}\bar{z},\;v_k\to
v\;\text{ as }\;k\to\infty\;\text{ with
}\;v_k\in\widehat{N}_\Omega(z_k)\big\}. \end{equation} Given further
a set-valued mapping $F\colon\R^n\tto\R^m$ with the graph
\begin{equation*} \gph F:=\big\{(x,y)\in\R^n\times\R^m\;\big|\;y\in
F(x)\big\}, \end{equation*} the (basic/limiting) {\em coderivative}
of $F$ at $(\ox,\oy)\in\gph F$ is defined via the limiting normal
cone \eqref{lnc} to the graph of $F$ at the reference point
$(\ox,\oy)$ as \begin{equation}\label{lim-cod}
D^*F(\ox,\oy)(v):=\big\{u\in\R^n\;\big|\;(u,-v)\in N_{{\rm
gph}\,F}(\ox,\oy)\big\},\quad v\in\R^m, \end{equation} where $\oy$
is omitted in the coderivative notation if $F(\bar{x})=\{\bar{y}\}$.
Note that if $F\colon\R^n\to\R^m$ is a (single-valued) ${\cal
	C}^1$-smooth mapping around $\ox$, then we have \begin{equation*}
D^*F(\bar{x})(v)=\big\{\nabla F(\bar{x})^*v\big\}\;\mbox{ for all
}\;v\in\R^m \end{equation*} in terms of the transpose matrix
(adjoint operator) $\nabla F(\bar{x})^*$ of the Jacobian $\nabla
F(\bar{x})$. Recall further that a set-valued mapping
$F:\R^n\rightrightarrows \R^m$ is {\it metrically regular} around
$(\bar{x},\bar{y}) \in \text{\rm gph}F$ with modulus $\mu>0$ if
there exist neighborhoods $U$ of $\bar{x}$ and $V$ of $\bar{y}$ such
that \begin{equation*} {\rm dist}\big(x;F^{-1}(y)\big)\le\mu\,{\rm
dist}\big(y;F(x)\big)\;\text{ for all }\;(x,y)\in U\times V,
\end{equation*} where $F^{-1}(y):=\{x\in\R^n\;|\;y\in F(x)\}$ is the
inverse mapping of $F$. If in addition $F^{-1}$ has a single-valued
localization around $(\bar{y},\bar{x})$, i.e., there exist some
neighborhoods $U$ of $\ox$ and $V$ of $\oy$ together with a
single-valued mapping $\vartheta: V\to U$  such that $\text{\rm
	gph}F^{-1}\cap (V\times U) = \text{\rm gph}\vartheta$, then $F$ is
{\it strongly metrically regular} around $(\bar{x},\bar{y})$ with
modulus $\mu>0$. A set-valued mapping $T: \R^n \rightrightarrows
\R^n$ is \textit{locally strongly monotone} with modulus $\tau  > 0$
around $(\ox, \oy) \in  \text{\rm gph}T$  if there exist
neighborhoods $U$ of $\ox$ and $V$ of $\oy$ such that $$ \langle
x-u, v- w\rangle \geq \tau \|x-u\|^2 \quad \text{for all }\; (x,v),
(u,w) \in \text{\rm gph}T\cap (U\times V). $$ If in addition
$\text{\rm gph}T \cap (U\times V) = \text{\rm gph}S \cap (U\times
V)$ for any monotone operator $S: \R^n\rightrightarrows \R^n$
satisfying the inclusion $\text{\rm gph}T \cap (U\times V)
\subset \text{\rm gph}S$, $T$ is called \textit{locally strongly maximally monotone} with modulus $\tau  > 0$ around $(\ox, \oy)$.\\[1ex]
Let $\ph\colon\R^n\to\oR$ be an extended-real-valued function with the domain and epigraph
\begin{equation*}
\dom\ph:=\big\{x\in\R^n\;\big|\;\ph(x)<\infty\big\}\;\mbox{ and }\;\epi\ph:=\big\{(x,\al)\in\R^{n+1}\;\big|\;\al\ge\ph(x)\big\}.
\end{equation*}
The (basic/limiting) {\em subdifferential} of $\ph$ at $\ox\in\dom\ph$ is defined geometrically
\begin{equation}\label{lim-sub}
\partial\varphi(\ox):=\big\{v\in\R^n\;\big|\;(v,-1)\in N_{{\rm\small epi}\,\varphi}\big(\bar{x},\varphi(\bar{x})\big)\big\}
\end{equation}
via the limiting normal cone \eqref{lnc}, while admitting various analytic representations. This subdifferential is an extension of the classical gradient for smooth functions and of the classical subdifferential of convex ones. If $F\colon\R^n\to\R^m$ is locally Lipschitzian around $\ox$, then we have the useful relationships between the coderivative \eqref{lim-cod} of $F$ and the subdifferential \eqref{lim-sub} of the scalarized function $\la v,F\ra(x):=\la v,F(x)\ra$ formulated as
\begin{equation}\label{scal}
D^*F(\bar{x})(v)=\partial\langle v,F\rangle(\bar{x})\;\mbox{ for all }\;v\in\R^m.
\end{equation}
Following \cite{m92}, we now define the {\em second-order subdifferential} $\partial^2\ph(\ox,\ov)\colon\R^n\tto\R^n$ of $\ph\colon\R^n\to\oR$ at $\ox\in\dom\ph$ for $\ov\in\partial\ph(\ox)$ as the coderivative \eqref{lim-cod} of the subgradient mapping \eqref{lim-sub}, i.e., by
\begin{equation}\label{2nd}
\partial^2\ph(\ox,\ov)(u):=\big(D^*\partial\ph\big)(\ox,\oy)(u)\;\mbox{ for all }\;u\in\R^n.
\end{equation}
If $\ph$ is ${\cal C}^2$-smooth around $\ox$, then we have the representation of the second-order subdifferential via the classical (symmetric) Hessian matrix
\begin{equation}\label{C2_Case}
\partial^2\ph(\ox)(u)=\big\{\nabla^2\ph(\ox)u\big\}\;\mbox{ for all }\;u\in\R^n,
\end{equation}
which allows us to also label \eqref{2nd} as the {\em generalized Hessian}.
In the case of ${\cal C}^{1,1}$ functions $\ph$, the second-order subdifferential \eqref{2nd} is computed by the scalarization formula \eqref{scal} via the coderivative of the gradient mapping $\nabla\ph$. In Section~\ref{intro}, the reader can find the references to many publications where the second-order subdifferential is computed entirely via the given data for major classes of systems appeared in variational analysis and optimization. It is important to mention that our basic constructions \eqref{lnc}--\eqref{lim-sub} and \eqref{2nd}, enjoy comprehensive {\em calculus rules} in general settings, despite being intrinsically nonconvex. This is due to {\em variational/extremal principles} of variational analysis; see the books \cite{Mordukhovich06,Mor18,Rockafellar98} for the first-order constructions and \cite{Mordukhovich06,Mor18} for the second-order subdifferential \eqref{2nd}.\vspace*{0.03in}

In what follows we are going to broadly employ the fundamental notion of {\em tilt stability} of local minimizers for extended-real-valued functions, which was introduced by Poliquin and Rockafellar \cite{Poli} and characterized therein in terms of the second-order subdifferential \eqref{2nd}.

\begin{Definition}[\bf tilt-stable local minimizers]\label{def:tilt} Given $\varphi\colon\R^n\to\oR$, a point $\ox\in\dom\varphi$ is a {\sc tilt-stable local minimizer} of $\varphi$ if there exists a number $\gamma>0$ such that the mapping
\begin{equation*}
M_\gamma\colon v\mapsto{\rm argmin}\big\{\varphi(x)-\langle v,x\rangle\;\big|\;x\in\B_\gamma(\ox)\big\}
\end{equation*}
is single-valued and Lipschitz continuous on a neighborhood of $0\in\R^n$ with $M_\gamma(0)=\{\ox\}$. By a {\sc modulus} of tilt stability of $\ph$ at $\ox$ we understand a Lipschitz constant of $M_\gamma$ around the origin.
\end{Definition}

Besides the seminal paper \cite{Poli}, the notion of tilt stability has been largely investigated, characterized, and widely applied in many publications to various classes of unconstrained and constrained optimization problems; see, e.g., \cite{ChieuNghia,dl,dmn,gm,Mor18,MorduNghia,mr} and the references therein.\vspace*{-0.05in}

\section{Globally Convergent GDNM in $\mathcal{C}^{1,1}$ Optimization}\label{sec:dampedC11}

In this section we concentrate on the unconstrained optimization problem \eqref{opproblem}, where the cost function $\varphi\colon\R^n\to\R$ is of class $\mathcal{C}^{1,1}$ around the reference points. The corresponding gradient equation associated with \eqref{opproblem}, which gives us, in particular, a necessary condition for local minimizers, is written in the form
\begin{equation}\label{gradientsys}
\nabla\varphi(x)=0.
\end{equation}

The following generalization of the pure Newton algorithm to solve \eqref{opproblem} {\em locally} was first suggested and investigated in \cite{BorisEbrahim} under the major assumption that a given point $\ox$ is a tilt-stable local minimizer of \eqref{opproblem}. Then it was extended in \cite{BorisKhanhPhat} to solve directly the gradient equation \eqref{gradientsys} under certain assumptions on a given solution $\ox$ to \eqref{gradientsys} ensuring the well-posedness and local superlinear convergence of the algorithm.

\begin{Algorithm}[\bf generalized pure Newton-type algorithm for ${\cal C}^{1,1}$ functions]\label{NM}\text{ }{\rm\\[1ex]
		{\bf Step~0:} Choose a starting point $x^0\in\R^n$ and set $k=0$.\\[1ex]
		{\bf Step~1:} If $\nabla\varphi(x^k)=0$, stop the algorithm. Otherwise move to Step~2.\\[1ex]
		{\bf Step~2:} Choose $d^k\in\R^n$ satisfying
		\begin{equation*}
		-\nabla\varphi(x^k)\in\partial^2\varphi(x^k)(d^k).
		\end{equation*}
		{\bf Step~3:} Set $x^{k+1}$ given by
		\begin{equation*}
		x^{k+1}:=x^k+d^k\;\mbox{ for all }\;k=0,1,\ldots.
		\end{equation*}
		{\bf Step~4:} Increase $k$ by $1$ and go to Step~1.}
\end{Algorithm}

One of the serious disadvantages of the pure Newton method and its generalizations is that the corresponding sequence of iterates may not converges if the stating point is not sufficiently close to the solution.  This motivates us to design and justify the {\em globally} convergent {\em damped Newton} counterpart of Algorithm~\ref{NM} with {\em backtracking line search} to solve the gradient equation \eqref{gradientsys} that is formalized as follows.

\begin{Algorithm}[\bf generalized damped Newton algorithm for $\mathcal{C}^{1,1}$ functions]\label{LS}\rm Let $\sigma\in\left(0,\frac{1}{2}\right)$ and $\beta\in\left(0,1\right)$ be given real numbers. Then do the following:\\[1ex]
	{\bf Step~0:} Choose an arbitrary staring point $x^0\in\R^n$ and set $k=0$.\\[1ex]
	{\bf Step~1:} If $\nabla\varphi(x^k)=0$, stop the algorithm. Otherwise move to Step~2.\\[1ex]
	{\bf Step~2:} Choose $d^k\in\R^n$ such that
	\begin{equation}\label{alC11}
	-\nabla\varphi(x^k)\in\partial^2\varphi(x^k)(d^k). \end{equation}
	{\bf Step~3:} Set $\tau_k=1$. Until \textit{Armijo's inequality} $$
	\varphi(x^k+\tau_kd^k)\leq
	\varphi(x^k)+\sigma\tau_k\langle\nabla\varphi(x^k),d^k\rangle. $$
	is satisfied, set $\tau_k:=\beta\tau_k$.\\[1ex]
	{\bf Step~4:} Set $x^k$ given by
	$$
	x^{k+1}:=x^k+\tau_k d^k\;\mbox{ for all }\;k=0,1,\ldots.
	$$
	{\bf Step~5:} Increase $k$ by $1$ and go to Step~1.
	
\end{Algorithm}

Due to \eqref{C2_Case}, Algorithm~\ref{LS} reduces to the standard damped Newton method (as, e.g., in \cite{BeckAl,Boyd}) if $\varphi$ is $\mathcal{C}^2$-smooth. Note also that by \eqref{lim-cod} the direction $d^k$ in \eqref{alC11} can be explicitly found from
\begin{equation*}
\big(-\nabla\varphi(x^k),-d^k\big)\in N\big((x^k,\nabla\varphi(x^k));\gph\nabla\varphi\big).
\end{equation*}
To proceed with the study of Algorithm~\ref{LS}, first we clarify the existence of {\em descent} Newton directions. It is done in the next proposition under the {\em positive-definiteness} of the second-order subdifferential mapping $\partial^2\ph(x)$.

\begin{Proposition}[\bf existence of descent Newton directions]\label{descent} Let $\varphi\colon\R^n\to\R$ be of class $\mathcal{C}^{1,1}$ around $x\in\R^n$. Suppose that $\nabla\varphi(x)\ne 0$ and that $\partial^2\varphi(x)$ is positive-definite, i.e.,
	\begin{equation}\label{pos-def}
	\langle z,u\rangle>0\;\mbox{ for all }\;z\in\partial^2\varphi(x)(u)\;\mbox{ and }\;u\ne 0.
	\end{equation}
	Then there exists a nonzero direction $d\in\R^n$ such that
	\begin{equation}\label{descentforvaphi}
	-\nabla\varphi(x)\in\partial^2\varphi(x)(d).
	\end{equation}
	Moreover, every such direction satisfies the inequality $\langle
	\nabla\varphi(x),d\rangle <0$. Consequently, for each
	$\sigma\in(0,1)$ and $d\in\R^n$ satisfying \eqref{descentforvaphi}
	we have $\delta>0$ such that
	\begin{equation}\label{Armijo}
	\varphi(x+\tau d)\le\varphi(x)+\sigma\tau\langle\nabla\varphi(x),d\rangle\;\mbox{ whenever }\;\tau\in(0,\delta).
	\end{equation}
\end{Proposition}
\begin{proof} It follows from \cite[Theorem~5.16]{Mor18} that $\nabla\varphi$ is strongly locally maximal monotone around $(x,\nabla\varphi(x))$. Thus $\nabla\varphi$ is strongly metrically regular
	around $(x,\nabla\varphi(x))$ by \cite[Theorem~5.13]{Mor18}. Using  {\cite[Theorem~4.2]{BorisKhanhPhat}} yields the existence of $d\in\R^n$ with $-\nabla\varphi(x)\in\partial^2\varphi(x)(d)$. To verify that $d\ne 0$, suppose on the contrary that $d=0$. Since $\nabla\varphi$ is locally Lipschitz around $x$, it follows from \cite[Theorem~1.44]{Mordukhovich06} that
	$$
	\partial^2\varphi(x)(d)=\big(D^*\nabla\varphi\big)(x)(d)=\big(D^*\nabla\varphi\big)(x)(0)=\{0\}.
	$$
	Therefore, we have that $\nabla\varphi(x)=0$ due to the inclusion $-\nabla\varphi(x)\in\partial^2\varphi(x)(d)$, which contradicts the assumption that $\nabla\varphi(x)\ne 0$. Employing the imposed positive-definiteness of the mapping $\partial^2\varphi(x)$ tells us that $\langle\nabla\varphi(x),d\rangle<0$. Using finally \cite[Lemmas~2.18 and 2.19]{Solo14}, we arrive at
	\eqref{Armijo} and thus complete the proof of the proposition.
\end{proof}

Now we formulate and discuss our major assumption to establish the desired global behavior of Algorithm~\ref{LS} for $\mathcal{C}^{1,1}$ functions $\ph$. Fix an arbitrary point $x^0\in\R^n$ and consider the level set
\begin{equation}\label{level}
\Omega:=\big\{x\in\R^n\;\big|\;\varphi(x)\le\varphi(x^0)\big\}.
\end{equation}

\begin{Assumption}\label{PDofHessian} \rm The mapping $\partial^2\varphi(x)\colon\R^n\tto\R^n$ is positive-definite \eqref{pos-def} for all $x\in\Omega$.
\end{Assumption}

Observe that Assumption~\ref{PDofHessian} {\em cannot be removed} or
even {\em replaced} by the {\em positive-semidefiniteness} of
$\partial^2\ph(x)$ to ensure the existence of descent Newton
direction for Algorithm~\ref{LS} as in Proposition~\ref{descent}.
Indeed, consider the simplest linear function $\varphi(x):=x$ on
$\R$. Then we obviously have that $\nabla^2\varphi(x)\ge 0$ for all
$x\in\R$, while there is no Newtonian direction $d\in\R$ satisfying
the backtracking line search condition
\eqref{Armijo}.\vspace*{0.05in}

The next theorem shows that Assumption~\ref{PDofHessian} not only ensures the {\em well-posedness} of Algorithm~\ref{LS}, but also allows us to conclude that all the limiting points of the iterative sequence $\{x^k\}$ are {\em tilt-stable minimizers} of the cost function $\ph$.

\begin{Theorem}[\bf well-posedness and limiting points of the generalized damped Newton algorithm]\label{globalconver} Let $\varphi\colon\R^n\to\R$ be of class $\mathcal{C}^{1,1}$, and let $x^0\in\R^n$ be an arbitrary point such that Assumption~{\rm\ref{PDofHessian}} is satisfied. Then we have the following assertions:
	\begin{itemize}
		\item[\bf(i)] Any sequence $\{x^k\}$ generated by Algorithm~{\rm\ref{LS}} is well-defined with $x^k\in\Omega$ for all $k\in\N$.
		\item[\bf(ii)] All the limiting points of $\{x^k\}$ are tilt-stable local minimizers of $\varphi$.
	\end{itemize}
\end{Theorem}
\begin{proof} First we check that a sequence $\{x^k\}$ generated by Algorithm~\ref{LS} with any starting point $x^0$ is well-defined. Indeed, there is nothing to prove if $\nabla\varphi(x^0)=0$. Otherwise, it follows from Proposition~\ref{descent} due to the positive-definiteness of $\partial^2\varphi(x^0)$ that there exist $d^0$ and $\tau_0$ satisfying $-\nabla\varphi(x^0)\in\partial^2\varphi(x^0)(d^0)$ and the inequalities
	$$
	\varphi(x^1)\le\varphi(x^0)+\sigma\tau_0\langle\nabla\varphi(x^0),d^0\rangle<\varphi(x^0),
	$$
	which clearly ensure that $x^1\in\Omega$. Then we get by induction that either $x^k\in\Omega$, or $\nabla\varphi(x^k)=0$ whenever $k\in\N$. Thus assertion (i) is verified.
	
	Next we prove assertion (ii). To proceed, suppose that $\{x^k\}$ has a limiting point $\bar{x}\in\R^n$, i.e., there exists a subsequence $\{x^{k_j}\}_{j\in\N}$ of $\{x^k\}$ such that $x^{k_j}\to \bar{x}$ as $j\to\infty$. Since $\Omega$ is closed and since $x^{k_j}\in\Omega$ for all $j\in\N$, we get $\bar{x}\in\Omega$. By Assumption~\ref{PDofHessian}, the mapping $\partial^2\varphi(\bar{x})$ is positive-definite. Then \cite[Proposition~4.6]{ChieuLee17} gives us positive numbers $\kappa$ and $\delta$ such that
	\begin{equation}\label{uniformPD}
	\langle z,w\rangle\ge\kappa\|w\|^2\quad\text{for all }\;z\in\partial^2\varphi(x)(w),\;x\in\mathbb{B}_\delta(\bar{x}),\;\mbox{ and }\;w\in\R^n.
	\end{equation}
	Since $\ph$ is of class ${\cal C}^{1,1}$ around $\bar{x}$, we get without loss of generality that $\nabla\varphi$ is Lipschitz continuous on $\mathbb{B}_\delta(\bar{x})$ with some constant $\ell>0$. The rest of the proof is split into the following two claims.\\[1ex]
	{\bf Claim~1:} {\em The sequence $\{\tau_{k_j}\}_{j\in\N}$ in Algorithm~{\rm\ref{LS}} is bounded from below by a positive number $\gamma>0$}. Indeed, suppose on the contrary that the statement does not hold. Combining this with $\tau_k\ge 0$ gives us a subsequence of $\{\tau_{k_j}\}$ that converges to $0$. Suppose without loss of generality that $\tau_{k_j}\to 0$ as $j\to\infty$. Since $-\nabla \varphi(x^{k_j})\in\partial^2\varphi(x^{k_j})(d^{k_j})$ for all $j\in\N$, we deduce from \eqref{uniformPD} and the Cauchy-Schwarz inequality that
	$$
	\|\nabla\varphi(x^{k_j})\|\ge\kappa\|d^{k_j}\|\;\mbox{ whenever }\;j\in\N,
	$$
	which verifies the boundedness of $\{d^{k_j}\}$. Thus $x^{k_j}+\beta^{-1}\tau_{k_j}d^{k_j}\to\bar{x}$ as $j\to\infty$, and hence
	$$
	x^{k_j}+\beta^{-1}\tau_{k_j}d^{k_j}\in\mathbb{B}_\delta(\bar{x})
	$$
	for all $j\in\N$ sufficiently large. Since $\ph$ is of class ${\cal C}^{1,1}$ around $\bar{x}$, we suppose without loss of generality that $\nabla\varphi$ is Lipschitz continuous on $\mathbb{B}_\delta(\bar{x})$. It follows then from \cite[Lemma~A.11]{Solo14} that
	\begin{equation}\label{Lipineq}
	\varphi(x^{k_j}+\beta^{-1}\tau_{k_j}d^{k_j})\le\varphi(x^{k_j})+\beta^{-1}\tau_{k_j}\langle\nabla\varphi(x^{k_j}),d^{k_j}\rangle +\frac{\ell\beta^{-2}\tau_{k_j}^2}{2}\|d^{k_j}\|^2
	\end{equation}
	whenever indices $j\in\N$ are sufficiently large. Due to the exit condition of the backtracking line search in Step~3 of Algorithm~\ref{LS}, we have the strict inequality
	\begin{equation}\label{exitcond}
	\varphi(x^{k_j}+\beta^{-1}\tau_{k_j}d^{k_j})>\varphi(x^{k_j})+\sigma\beta^{-1}\tau_{k_j}\langle\nabla \varphi(x^{k_j}),d^{k_j}\rangle
	\end{equation}
	for large $j\in\N$. Now combining \eqref{uniformPD}, \eqref{Lipineq}, and \eqref{exitcond} for such $j$ yields the estimates
	\begin{eqnarray*}
		\sigma\beta^{-1}\tau_{k_j}\langle\nabla\varphi(x^{k_j}),d^{k_j}\rangle&<&\beta^{-1}\tau_{k_j}\langle\nabla\varphi(x^{k_j}),d^{k_j}\rangle+\frac{\ell\beta^{-2}\tau_{k_j}^2}{2}\|d^{k_j}\|^2\\
		&\le&\beta^{-1}\tau_{k_j}\langle\nabla\varphi(x^{k_j}),d^{k_j}\rangle+\frac{\ell\beta^{-2}\tau_{k_j}^2}{2\kappa}\langle\nabla\varphi(x^{k_j}),-d^{k_j}\rangle,
	\end{eqnarray*}
	which imply in turn that $\sigma\beta>\beta-\frac{\ell}{2\kappa}\tau_{k_j}$ for all large $j\in\N$. Letting $j\to\infty$ gives us $\sigma\beta\ge\beta$, a contradiction due to $\sigma<1$ and $\beta>0$. This justifies the claimed boundedness of $\{\tau_{k_j}\}_{j\in\N}$.\\[1ex]
	{\bf Claim~2:} {\em Any limiting point $\bar{x}$ of $\{x^k\}$ is a tilt-stable local minimizer of $\varphi$.} Indeed, it follows from the continuity of the gradient $\nabla\varphi$ and the convergence $x^{k_j}\to\bar{x}$ that $\nabla\varphi(x^{k_j})\to\nabla\varphi(\bar{x})$ as $j\to\infty$. Since the sequence $\{\varphi(x^k)\}$ is nonincreasing, we get that $\varphi(\bar{x})$ is a lower bound for $\{\varphi(x^k)\}$. Thus the sequence $\{\varphi(x^k)\}$ must converge to $\varphi(\bar{x})$ as $k\to\infty$. It follows from \cite[Theorem~1.44]{Mordukhovich06} due to $-\nabla\varphi(x^{k_j})\in\partial^2\varphi(x^{k_j})(d^{k_j})$ and the Lipschitz continuity of $\nabla\varphi$ on $\mathbb{B}_\delta(\bar{x})$ with constant $\ell$ that
	\begin{equation}\label{bounded1Lip}
	\|\nabla\varphi(x^{k_j})\|\le\ell\|d^{k_j}\|\;\text{ for sufficiently large }\;j\in\N.
	\end{equation}
	Combining Claim 1 with the estimates in \eqref{uniformPD} and \eqref{bounded1Lip}, we find $j_0\in\N$ such that
	\begin{equation}\label{ineq}
	\varphi(x^{k_j})-\varphi(x^{k_j+1})\ge\sigma\tau_{k_j}\langle-\nabla\varphi(x^{k_j}),d^{k_j}\rangle\ge\sigma\gamma\kappa\|d^{k_j}\|^2\ge\sigma\gamma\kappa\ell^{-2}\|\nabla\varphi(x^{k_j})\|^2,\;j\ge j_0.
\end{equation}
Since $\{\varphi(x^k)\}$ is convergent, it follows that the sequence $\{\varphi(x^{k_j})-\varphi(x^{k_j+1})\}_{j\in\N}$ converges to $0$ as $j\to\infty$. Furthermore, we deduce from $\eqref{ineq}$ that $\{\|\nabla\varphi(x^{k_j})\|\}$ also converges to $0$, and therefore $\nabla\varphi(\bar{x})=0$. Combining the latter with the positive-definiteness of $\partial^2\varphi(\bar{x})$ tells us by \cite[Theorem~1.3]{Poli} that $\bar{x}$ is a tilt-stable local minimizer of $\varphi$. This completes the proof.
\end{proof}

\begin{Remark}[\bf iterative sequences may diverge]\label{diver} \rm  Note that Theorem~\ref{globalconver} does not claim anything about the convergence of the iterative sequence $\{x^k\}$. In fact, the divergence of such a sequence can be observe in simple situations under the fulfillment of all the assumptions of Theorem~\ref{globalconver}. To illustrate it, consider the function $\varphi(x):=e^x$ on $\R$ with the positive second derivative $\varphi^{\prime\prime}(x)=e^x>0$ for all $x\in\R$. Running Algorithm~\ref{LS} with the starting point $x^0=1$, it is not hard to check that $\tau_k=1$ and $d^k=-1$ for all $k\in\N$. Thus the sequence of $x^k=1-k$ as $k\in\N$ generated by Algorithm~\ref{LS} is obviously divergent.
\end{Remark}

We conclude this section by giving a simple additional condition to Assumption~\ref{PDofHessian} that ensures the global convergence of any sequence of iterates in Algorithm~\ref{LS}.

\begin{Assumption}\label{boundedoflevelset} The level set $\Omega$ from \eqref{level} is bounded.
\end{Assumption}

To establish the global convergence of Algorithm~\ref{LS}, we first present the following useful lemma of its own interest.

\begin{Lemma}[\bf uniformly positive-definiteness of second-order subdifferentials]\label{uniPD} \text{}
	Let $\varphi\colon\R^n\to\R$ be a function of class $\mathcal{C}^{1,1}$, and let $x^0$ be arbitrary vector for which Assumptions~{\rm\ref{PDofHessian}} and {\rm\ref{boundedoflevelset}} are satisfied. Then there exists $\kappa>0$ such that for each $x\in\Omega$ we have
	\begin{equation}\label{unisecond-order}
	\langle z,w\rangle\ge\kappa\|w\|^2\;\mbox{ whenever }\;z\in\partial^2\varphi(x)(w)\;\mbox{ and }\;w\in\R^n.
	\end{equation}
\end{Lemma}
\begin{proof} Since the mapping $\partial^2\varphi(x)$ is positive-definite for each $x\in\Omega$ by Assumption~\ref{PDofHessian}, we deduce from \cite[Proposition~4.6]{ChieuLee17} that there
	exist $\kappa_x>0$ and a neighborhood $U_x$ of $x$ such that
	\begin{equation}\label{uniformPDx}
	\langle z,w\rangle\ge\kappa_x\|w\|^2\;\text{ for all }\;z\in\partial^2\varphi(y)(w),\;y\in U_x,\;\mbox{ and }\;w\in\R^n.
	\end{equation}
	Note that the neighborhood system $\{U_x\;|\;x\in\Omega\}$ provides an open cover of $\Omega$. Using the compactness of the set $\Omega$ due its closedness and Assumption~\ref{boundedoflevelset}, we find finitely many points $x_1,\ldots,x_p\in\Omega$ such that $\Omega\subset\bigcup_{j=1}^p U_{x_j}$. Denoting
	$$
	\kappa:={\rm min}\big\{\kappa_{x_1},\ldots,\kappa_{x_p}\big\}>0,
	$$
	we arrive at the fulfillment of the claimed condition \eqref{unisecond-order} for each $x\in\Omega$.
\end{proof}

Now we are ready to justify the global convergence of Algorithm~\ref{LS}.

\begin{Theorem}[\bf global convergence of the damped Newton algorithm for ${\cal C}^{1,1}$ functions]\label{sufconver} In the setting of Theorem~{\rm\ref{globalconver}}, suppose in addition that Assumption~{\rm\ref{boundedoflevelset}} is satisfied. Then the sequence $\{x^k\}$ is convergent, and its limit is a tilt-stable local minimizer of $\varphi$.
\end{Theorem}
\begin{proof} The well-definiteness of the sequence $\{x^k\}$ and the inclusion $\{x^k\}\subset\Omega$ follow from Theorem~\ref{globalconver}. Furthermore, employing Assumptions~\ref{PDofHessian} and \ref{boundedoflevelset}, the inclusion $-\nabla\varphi(x^k)\in\partial^2\varphi(x^k)(d^k)$ for all $k\in\N$, and Proposition~\ref{uniPD} ensures the existence of $\kappa>0$ such that
	\begin{equation}\label{boundeduni}
	\langle-\nabla \varphi(x^{k}),d^{k}\rangle\ge\kappa\|d^{k}\|^2\quad\text{for all }\;k\in\N.
	\end{equation}
	Assumption~\ref{boundedoflevelset} tells us that the sequence $\{x^k\}$ is bounded, and so it has a limiting point $\bar{x}\in\Omega$. Hence the value $\varphi(\bar{x})$ is a limiting point of the numerical sequence $\{\varphi(x^k)\}$. Combining this with the nonincreasing property of $\{\varphi(x^k)\}$ yields the convergence of $\{\varphi(x^k)\}$ to $\varphi(\bar{x})$ as $k\to\infty$. It follows from \eqref{boundeduni} that
	\begin{equation}\label{ineq1}
	\varphi(x^{k})-\varphi(x^{k+1})\ge\sigma\tau_{k}\langle-\nabla\varphi(x^{k}),d^{k}\rangle\ge\sigma\tau_k\kappa\|d^{k}\|^2\;\text{ for all }\;k\in\N.
	\end{equation}
	The above convergence of $\{\varphi(x^k)\}$ implies that the sequence $\{\varphi(x^{k})-\varphi(x^{k+1})\}_{k\in\N}$ converges to $0$ as $k\to\infty$. It follows from $\eqref{ineq1}$ that
	\begin{equation}\label{gradientconverge}
	\lim_{k\to\infty}\tau_k\|d^k\|^2=0.
	\end{equation}
Now let us show that the sequence $\{x^k\}$ converges to $\bar{x}$ as $k\to\infty$ by using Ostrowski's condition from \cite[Proposition~8.3.10]{JPang}. To accomplish this, we verify that there is a neighborhood of $\bar{x}$ within which no other limiting point of $\{x^k\}$ exists, and the following condition holds:
	\begin{equation}\label{Ostrowski}
	\lim_{k\to\infty}\|x^{k+1}-x^k\|=0.
	\end{equation}
	Indeed, tilt stability of the local minimizer $\bar{x}$ of $\ph$ ensures the existence of $\delta>0$ for which $\varphi$ is strongly convex on $\mathbb{B}_\delta(\bar{x})$ due to \cite[Theorem~4.7]{ChieuLee17}. Arguing by contraposition, suppose that there is $\Tilde x\in\mathbb{B}_\delta(\bar{x})$ such that $\Tilde x\ne \bar{x}$ and $\Tilde x$ is a limiting point of $\{x^k\}$. Theorem~\ref{globalconver} tells us that $\Tilde x$ is also a tilt-stable local minimizer of $\varphi$, a contradiction with the strong convexity of $\varphi$ on $\mathbb{B}_\delta(\bar{x})$. Moreover, the construction of $\{x^k\}$ and the condition $\tau_{k}\in(0,1]$ imply the estimate
	$$
	\|x^{k+1}-x^k\|^2=\tau_k^2\|d^k\|^2\le\tau_k\|d^k\|^2\;\mbox{ for all }\;k\in\N.
	$$
	Passing there to the limit as $k\to\infty$ and using \eqref{gradientconverge}, we verify \eqref{Ostrowski}. Finally, it follows from \cite[Proposition~8.3.10]{JPang} that $\{x^k\}$ converges to $\bar{x}$, which thus completes the proof.
\end{proof}

\section{Rates of Convergence of GDNM for ${\cal C}^{1,1}$ Optimization}\label{sec:rate}

This section is devoted to the study of {\em convergence rates} of the globally convergent Algorithm~\ref{LS} for minimizing ${\cal C}^{1,1}$ functions. First we recall the standard notions of convergence rates used in what follows; see \cite[Definition~7.2.1]{JPang}.

\begin{Definition}[\bf rates of convergence]\label{def-rates}\rm Let $\{x^k\}\subset\R^n$ be a sequence of vectors converging to $\ox$ as $k\to\infty$ with $\bar{x}\ne x^k$ for all $k\in\N$. The convergence rate is said to be (at least):
	\begin{itemize}
		\item[\bf(i)] \textsc{R-linear} if
		$$
		0<\limsup_{k\to\infty}\left(\|x^k-\bar{x}\|\right)^{1/k}<1,
		$$
		i.e., there exist $\mu\in(0,1)$, $c>0$, and $k_0\in\N$ such that
		$$
		\|x^{k}-\bar{x}\|\le c\mu^k\quad\text{for all }\;k\ge k_0.
		$$
		\item[\bf(ii)] \textsc{Q-linear} if
		$$
		\limsup_{k\to\infty}\frac{\|x^{k+1}-\bar{x}\|}{\|x^k-\bar{x}\|}<1,
		$$
		i.e., there exist $\mu\in(0,1)$ and $k_0\in\N$ such that
		$$
		\|x^{k+1}-\bar{x}\|\le\mu\|x^k-\bar{x}\|\quad\text{for all }\;k\ge k_0.
		$$
		\item[\bf(iii)] \textsc{Q-superlinear} if
		$$
		\lim_{k\to\infty}\frac{\|x^{k+1}-\bar{x}\|}{\|x^k-\bar{x}\|}=0.
		$$
	\end{itemize}
\end{Definition}

Our first result here establishes the {\em linear convergence} of Algorithm~\ref{LS} under the general assumptions formulated in the preceding section.

\begin{Theorem}[\bf linear convergence of generalized damped Newton algorithm for ${\cal C}^{1,1}$ functions]\label{linearcon} In the setting of Theorem~{\rm\ref{globalconver}}, suppose in addition that the sequence $\{x^k\}$ converges to some vector $\bar{x}$ being such that $x^k\ne\bar{x}$ for all $k\in\N$. Then we have the following assertions:
	\begin{itemize}
		\item[\bf(i)] The sequence $\{\varphi(x^k)\}$ converges to $\varphi(\bar{x}) $ at least Q-linearly.
		\item[\bf(ii)] The sequences $\{x^k\} $ and  $\{\|\nabla\varphi(x^k)\| \}$  converge  to $\bar{x}$ and $0$, respectively at least R-linearly.
	\end{itemize}
\end{Theorem}
\begin{proof} Suppose that $\{x^k\}$ converges to $\bar{x}$. It follows from Theorem~\ref{globalconver} that $\bar{x}$ is a tilt-stable local minimizer of $\varphi$. Using now the characterizations of tilt-stable local minimizers taken from \cite[Theorem~4.7]{ChieuLee17}, we deduce that there exist $\kappa>0$ and $\delta>0$ such that $\varphi$ is strongly convex on $\mathbb{B}_\delta(\bar{x})$ with modulus $\kappa$ while satisfying
	\begin{equation}\label{uniformPD1}
	\langle z,w\rangle\ge\kappa\|w\|^2\quad\text{for all }\;z\in\partial^2\varphi(x)(w),\;x\in\mathbb{B}_\delta(\bar{x}),\;\mbox{ and }\;w\in\R^n.
	\end{equation}
	Furthermore, due to the  locally Lipschitz continuity around $\bar{x}$ of $\nabla\varphi$, we suppose without loss of generality that $\nabla\varphi$ is Lipschitz continuous on $\mathbb{B}_\delta(\bar{x})$ with some constant $\ell>0$. The strong convexity of $\varphi$ on $\mathbb{B}_\delta(\bar{x})$ tells us that
	\begin{equation}\label{second-growth}
	\varphi(x)\ge\varphi(u)+\langle\nabla\varphi(u),x-u\rangle+\frac{\kappa}{2}\|x-u\|^2\;\mbox{ and}
	\end{equation}
	\begin{equation}\label{strongineq1}
	\langle\nabla\varphi(x)-\nabla \varphi(u),x-u\rangle\ge\kappa\|x-u\|^2\;\mbox{ for all }\;x,u\in\mathbb{B}_\delta(\bar{x}).
	\end{equation}
	Since $x^k\to\bar{x}$ as $k\to\infty$, we have that $x^k\in U$ for all $k\in\N$ sufficiently large. Substituting $x:=x^k$ and $u:=\bar{x}$ into \eqref{second-growth} and \eqref{strongineq1}, and then using the Cauchy-Schwarz inequality together with $\nabla\varphi(\bar{x})=0$ yield the estimates
	\begin{equation}\label{second-growth1}
	\varphi(x^k)\ge\varphi(\bar{x})+\frac{\kappa}{2}\|x^k-\bar{x}\|^2\;\mbox{ and}
	\end{equation}
	\begin{equation}\label{strongineq}
	\|\nabla\varphi(x^k)\|\ge\kappa\|x^k-\bar{x}\|
	\end{equation}
	for all large $k\in\N$. The local Lipschitz continuity of $\nabla\varphi$ around $\bar{x}$ and the result of \cite[Lemma~A.11]{Solo14} ensure the existence of $\ell>0$ such that
	\begin{equation}\label{Lipschitzineq}
	\varphi(x^k)-\varphi(\bar{x})=|\varphi(x^k)-\varphi(\bar{x})-\langle\nabla\varphi(\bar{x}),x^k-\bar{x}\rangle|\le\frac{\ell}{2}\|x^k-\bar{x}\|^2\quad\text{for large }\;k\in\N.
	\end{equation}
	Furthermore, the Lipschitz continuity of $\nabla\varphi$ on $\mathbb{B}_\delta(\bar{x})$ together with the inclusion $-\nabla\varphi(x^{k})\in\partial^2\varphi(x^{k})(d^{k})$ implies by \cite[Theorem~1.44]{Mordukhovich06} that
	\begin{equation}\label{bounded1Lip1}
	\|\nabla\varphi(x^{k})\|\le\ell\|d^{k}\|\quad\text{for large }\;k\in\N.
	\end{equation}
	Proceeding similarly to the proof of Theorem~\ref{globalconver} tells us that the sequence $\{\tau_k\}$ is bounded from below by some constant $\gamma>0$. Combining the latter with \eqref{uniformPD1} and \eqref{bounded1Lip1} yields
	\begin{equation}\label{ineq2}
	\varphi(x^{k})-\varphi(x^{k+1})\ge\sigma\tau_{k}\langle-\nabla\varphi(x^{k}),d^{k}\rangle\ge\sigma\gamma\kappa\|d^{k}\|^2\ge\sigma\gamma\kappa\ell^{-2}\|\nabla\varphi(x^{k})\|^2
	\end{equation}
	for all large $k\in\N$. Therefore, for such $k$ we deduce from \eqref{strongineq}, \eqref{Lipschitzineq}, and \eqref{ineq2} that
	\begin{equation*}
	\varphi(x^{k+1})-\varphi(x^k)\le-\sigma\gamma\kappa\ell^{-2}\|\nabla\varphi(x^{k})\|^2\le-\sigma\gamma\kappa^3\ell^{-2}\|x^k-\bar{x}\|^2\le-2\sigma\gamma\kappa^{3}\ell^{-3}\big(\varphi(x^k)
	-\varphi(\bar{x})\big).
	\end{equation*}
	This allows us to find $k_0\in\N$ such that
	$$
	\varphi(x^{k+1})-\varphi(\bar{x})\le\mu\big(\varphi(x^k)-\varphi(\bar{x})\big)\quad\text{whenever }\;k\ge k_0,
	$$
	which verifies (i) with $\mu:=1-2\sigma\gamma\kappa^{3}\ell^{-3}\in(0,1)$. It readily follows from \eqref{second-growth1} that
	$$
	\sqrt{\frac{2}{\kappa}\big(\varphi(x^k)-\varphi(\bar{x})\big)}\le\sqrt{\frac{2\mu}{\kappa}\big(\varphi(x^{k-1})-\varphi(\bar{x})\big)}\le\ldots\le\sqrt{\frac{2\mu^{k-k_0}}{\kappa}\big(\varphi(x^{k_0})
		-\varphi(\bar{x})\big)}
	$$
	whenever $k\ge k_0$. Thus we get $\|x^k-\bar{x}\|\le M\lambda^k$ for such $k$, where
	$$
	M:=\sqrt{\frac{2}{\kappa}\mu^{-k_0}\big(\varphi(x^{k_0})-\varphi(\bar{x})\big)} \quad\text{and}\quad\lambda:=\sqrt{\mu}.
	$$
	Since $\lambda\in(0,1)$, it follows that $\displaystyle\lim_{k\to\infty}\lambda^k=0$, which implies that the sequence $\{x^k\}$ converges at least R-linearly to $\bar{x}$ as $k\to\infty$. Employing again the Lipschitz continuity of $\nabla\varphi$ around $\bar{x}$ with constant $\ell>0$, we arrive at
	$$
	\|\nabla\varphi(x^k)\|=\|\nabla\varphi(x^k)-\nabla\varphi(\bar{x})\|\le\ell\|x^k-\bar{x}\|\le\ell M\lambda^k\quad\text{for all }\;k\ge k_0,
	$$
	which verifies assertion (ii) and thus completes the proof of the theorem. \end{proof}

To proceed further with deriving verifiable conditions ensuring the Q-superlinear convergence of Algorithm~\ref{LS}, we need to recall some important notions from variational analysis. A crucial role in the theory and applications of nonsmooth Newton-type methods is played by a remarkable subclass of single-valued locally Lipschitzian mappings defined as follows; see the books \cite{JPang,Solo14,Klatte} with the commentaries and references therein. A mapping $f\colon\R^n\to\R^m$ is {\em semismooth} at $\bar{x}$ if it is locally Lipschitzian around this point and the limit
\begin{equation}\label{semi-sm}
\lim_{A\in\text{\rm co}\overline{\nabla}f(\bar{x}+tu')\atop u'\to u,t\downarrow 0}Au'
\end{equation}
exists for all $u\in\R^n$, where `co' stands for the convex hull of a set, and where $\overline{\nabla}f$ is given by
$$
\overline{\nabla}f(x):=\big\{A\in\R^{m\times n}\big|\;\exists\ x_k\overset{\Omega_{f}}\to x\;\text{ such that }\;\nabla f(x_k)\to A\big\},\quad x\in\R^n,
$$
with $\Omega_{f}:=\{x\in\R^n\;|\;f\;\text{ is differentiable at }\;x\}$. Quite recently \cite{Helmut}, the concept of semismoothness has been improved and extended to set-valued mappings. To formulate the latter notion, recall first the construction of the {\em directional limiting normal cone} to a set $\Omega\subset\R^s$ at $\oz\in\O$ in the direction $d\in\R^s$ introduced in \cite{gin-mor} by
\begin{equation}\label{dir-nc}
N_\Omega(\oz;d):=\big\{v\in\R^s\;\big|\;\exists\,t_k\dn 0,\;d_k\to d,\;v_k\to v\;\mbox{ with }\;v_k\in\widehat{N}_\Omega(\oz+t_k d_k)\big\}.
\end{equation}
It is obvious that \eqref{dir-nc} agrees with the limiting normal cone \eqref{lnc} for $d=0$. The {\em directional limiting coderivative} of $F\colon\R^n\tto\R^m$ at $(\ox,\oy)\in\gph F$ in the direction $(u,v)\in\R^n\times\R^m$ is defined in \cite{g} via \eqref{dir-nc} by the scheme
\begin{equation}\label{dir-cod}
D^*F\big((\ox,\oy);(u,v)\big)(v^*):=\big\{u^*\in\R^n\;\big|\;(u^*,-v^*)\in N_{\text{gph}\,F}\big((\ox,\oy);(u,v)\big)\big\},\;v^*\in\R^m.
\end{equation}
Using \eqref{dir-cod}, we come to the aforementioned property of set-valued mappings introduced in \cite{Helmut}.

\begin{Definition}[\bf semismooth$^*$ property of set-valued mappings]\label{semi*} A mapping $F\colon\R^n\tto\R^m$ is {\sc semismooth$^*$} at $(\bar{x},\bar{y})\in\gph F$ if whenever $(u,v)\in\R^n\times\R^m$ we have
	\begin{equation*}
	\langle u^*,u\rangle=\langle v^*,v\rangle\;\mbox{ for all }\;(v^*,u^*)\in\gph D^*F\big((\ox,\oy);(u,v)\big).
	\end{equation*}
\end{Definition}
Recall \cite{Helmut} that  the semismoothness$^*$ holds if the graph of $F\colon\R^n\tto\R^m$ is represented as a union of finitely many closed and convex sets as well as for normal cone mappings generated by convex polyhedral sets. Note also that the semismooth$^*$ property of single-valued locally Lipschitzian mappings $f\colon\R^n\to\R^m$ around $\ox$ agrees with the semismooth property \eqref{semi-sm} at this point provided that $f$ directionally differentiable at $\ox$.\vspace*{0.05in}

Prior to deriving a major theorem on the $Q$-superlinear convergence
of Algorithm~\ref{LS}, we present an important property of
$\mathcal{C}^{1,1}$ functions with semismooth$^*$ derivatives.

\begin{Proposition}[\bf directional estimate for functions with semismooth$^*$ derivatives]\label{Maratos} Let $\varphi\colon\R^n\to\R$ be continuously differentiable around $\bar{x}\in\R^n$ with
	$\nabla\varphi(\bar{x})=0$. Suppose that  $\nabla\varphi$ is locally
	Lipschitzian around this point with modulus $\ell>0$, and that $\nabla\varphi$ is
	semismooth$^*$ at this point. Assume further that a sequence $\{x^k\}$
	converges to $\bar{x}$ with $x^k\ne\bar{x}$ as $k\in\N$, and that a
	sequence $\{d^k\}$ is superlinearly convergent with respect to
	$\{x^k\}$, i.e., $\|x^k+d^k-\bar{x}\|= o(\|x^k-\bar{x}\|)$. Consider
	the following statements:
	\begin{itemize}
		\item[\bf(i)] $\nabla \varphi$ is directionally differentiable at $\bar{x}$, and there exists $\kappa>0$ such that $\langle\nabla\varphi(x^k),d^k\rangle\le-\frac{1}{\kappa}\|d^k\|^2$ for all
		sufficiently large $k \in \N$.
		
		\item[\bf(ii)] There exists $\kappa>0$ such that
		\begin{equation}\label{growthdk}
		\varphi(x^k+d^k)-\varphi(x^k)\le\langle \nabla
		\varphi(x^k+d^k),d^k\rangle-\frac{1}{2\kappa}\|d^k\|^2
		\end{equation}
		for all sufficiently large $k \in \N$.
	\end{itemize}
	Then we have the estimate
	\begin{equation}\label{backtrackinghold}
	\varphi(x^k+d^k)\le\varphi(x^k)+\sigma\langle\nabla\varphi(x^k),d^k\rangle\;\mbox{
		for all large }\;k\in\N
	\end{equation}
	provided that either {\rm(i)} holds  {and $\sigma\in(0, 1/2)$}, or {\rm
		(ii)} holds and $\sigma \in\left(0,1/(2\ell\kappa) \right)$
\end{Proposition}
\begin{proof}
	Suppose that (i) holds. The directional differentiability and
	semismoothness$^*$ of $\nabla \varphi$ at $\bar{x}$ ensures by
	\cite[Corollary~3.8]{Helmut} that $\nabla\varphi$ is semismooth at
	$\bar{x}$. Using now \cite[Proposition~8.3.18]{JPang}, we
	immediately get \eqref{backtrackinghold}. Suppose next that (ii) is
	satisfied and $\sigma\in\left(0, 1/(2\ell\kappa) \right)$. Since
	$\{d^k\}$ is superlinearly convergent with respect to $\{x^k\}$, by
	employing \cite[Lemma~7.5.7]{JPang} we have
	\begin{equation}\label{superlineark+1}
	\lim_{k\to\infty}\|x^k-\bar{x}\|/\|d^k\|=1,
	\end{equation}
	which ensures in turn that
	\begin{equation}\label{supdk}
	\|x^k+d^k-\bar{x}\|=o(\|d^k\|)\;\text{ as }\;k\to\infty.
	\end{equation}
	Then the statement in (ii) leads us to the inequalities
	\begin{eqnarray*}
		\varphi(x^k+d^k)-\varphi(x^k)-\sigma\langle\nabla\varphi(x^k),d^k\rangle&\le&\langle\nabla\varphi(x^k+d^k),d^k\rangle-\frac{1}{2\kappa}\|d^k\|^2-\sigma\langle\nabla\varphi(x^k),d^k\rangle\\
		&\le&\|\nabla\varphi(x^k+d^k)\|\cdot\|d^k\|-\frac{1}{2\kappa}\|d^k\|^2+\sigma\|\nabla\varphi(x^k)\|\cdot\|d^k\|\\
		&\le&\ell\|x^{k}+d^k-\bar{x}\|\cdot\|d^k\|-\frac{1}{2\kappa}\|d^k\|^2+\sigma\ell\|x^k-\bar{x}\|\cdot\|d^k\|\\
		&\le&\|d^k\|^2
		\left(\ell\frac{\|x^{k}+d^k-\bar{x}\|}{\|d^k\|}-\frac{1}{2\kappa}+\sigma\ell\frac{\|x^k-\bar{x}\|}{\|d^k\|}
		\right)
	\end{eqnarray*}
	for all large $k \in \N$. Finally, it follows from the inequality
	$\sigma<1/(2\ell\kappa)$, \eqref{superlineark+1}, and \eqref{supdk}
	that
	$$
	\varphi(x^k+d^k)-\varphi(x^k)-\sigma\langle\nabla\varphi(x^k),d^k\rangle\leq
	0 \quad \text{whenever } \; k\in\N \;  \text{ is sufficiently
		large},
	$$
	which readily justifies \eqref{backtrackinghold} and completes the
	proof of the proposition.
\end{proof}

Now we are ready to derive the main result of this section that
establishes the Q-superlinear convergence of Algorithm~\ref{LS}
under the imposed assumptions.

\begin{Theorem}[\bf superlinear convergence of the generalized
	damped Newton algorithm for ${\cal C}^{1,1}$
	functions]\label{superlinearLS} In the setting of
	Theorem~{\rm\ref{globalconver}}, suppose that $\{x^k\}$ converges
	$\bar{x}$ as $k\to\infty$ with $x^k\ne\bar{x}$ for all $k\in\N$, and
	that $\nabla\varphi$ is locally Lipschitzian around $\bar{x}$ with
	some constant $\ell>0$ being also
	semismooth$^*$ at this point. Then the convergence rate of
	$\{x^k\}$ is at least Q-superlinear if either one of the following
	two conditions holds: \begin{itemize} \item[\bf(i)] $\nabla\varphi$
		is directionally differentiable at $\bar{x}$. \item[\bf(ii)]
		$\sigma\in\left(0,1/(2\ell\kappa)\right)$, where $\kappa>0$ is a
		modulus of tilt stability of $\ph$ at $\ox$. \end{itemize}
	Furthermore, in both cases {\rm(i)} and {\rm(ii)} the sequence
	$\{\varphi(x^k)\}$ converges Q-superlinearly to $\varphi(\bar{x})$,
	and the sequence $\{\nabla\varphi(x^k)\}$ converges Q-superlinearly
	to $0$ as $k\to\infty$. \end{Theorem}

\begin{proof} Suppose that the sequence $\{x^k\}$ converges as $k\to\infty$ to a point $\bar{x}\in\R^n$. We split the proof of the theorem into the following three claims.\\[1ex]
	{\bf Claim~1:} {\em The sequence $\{d^k\}$ is superlinearly
		convergent with respect to $\{x^k\}$}. Indeed, it follows from
	Theorem~\ref{globalconver} that $\ox$ is tilt-stable local minimizer
	of $\ph$ with some modulus $\kappa>0$. By using the characterization
	of tilt-stable minimizers via the combined second-order
	subdifferential taken from \cite[Theorem~3.5]{MorduNghia} and
	\cite[Proposition 4.6]{ChieuLee17}, we find a positive number
	$\delta$ such that \begin{equation}\label{uniformPDrate} \langle
	z,w\rangle\ge\frac{1}{\kappa}\|w\|^2\quad\text{for all
	}\;z\in\partial^2\varphi(x)(w),\;x\in\mathbb{B}_\delta(\bar{x}),\;\mbox{
		and }\;w\in\R^n. \end{equation} Using the subadditivity property of
	coderivatives obtained in \cite[Lemma~5.6]{BorisKhanhPhat}, we have
	$$
	\partial^2\varphi(x^k)(d^k)\subset\partial^2\varphi(x^k)(x^k+d^k-\bar{x})+\partial^2\varphi(x^k)(-x^k+\bar{x}).
	$$ Since $-\nabla\varphi(x^k)\in\partial^2\varphi(x^k)(d^k)$, for
	all $k\in\N$ there exists
	$v^k\in\partial^2\varphi(x^k)(-x^k+\bar{x})$ such that $$
	-\nabla\varphi(x^k)-v^k\in\partial^2\varphi(x^k)(x^k+d^k-\bar{x}).
	$$ Employing \eqref{uniformPDrate} and the Cauchy-Schwarz
	inequality, we have \begin{equation}\label{semi}
	\|x^k+d^k-\bar{x}\|\le\kappa\|\nabla\varphi(x^k)+v^k\|\quad\text{for
		sufficiently large }\;k\in\N. \end{equation} By the
	semismoothness$^*$ of $\nabla\varphi$ at $\bar{x}$  together with
	$\nabla\varphi(\bar{x})=0$  and \cite[Lemma~5.5]{BorisKhanhPhat}, we
	have \begin{equation}\label{semi2}
	\|\nabla\varphi(x^k)+v^k\|=\|\nabla\varphi(x^k)-\nabla\varphi(\bar{x})+v^k\|=o(\|x^k-\bar{x}\|).
	\end{equation}
	Combining \eqref{semi} and \eqref{semi2} tells us that $\|x^k+d^k-\bar{x}\|=o(\|x^k-\bar{x}\|)$ as $k\to\infty$, which thus completes the proof of this claim.\\[1ex]
	{\bf Claim~2:} {\em We have $\tau_k=1$ for all $k\in\N$ sufficiently
		large provided that either condition {\rm(i)} or {\rm(ii)} of this
		theorem holds}. To proceed, we need to show that for all $k\in\N$
	sufficiently large, we get the inequality \eqref{backtrackinghold}.
	Assume first that (i) holds. Due to the inclusion
	$-\nabla\varphi(x^k)\in\partial^2\varphi(x^k)(d^k)$ for all $k \in
	\N$ and the estimate \eqref{uniformPDrate}, we have
	$\langle\nabla\varphi(x^k),d^k\rangle\le -\frac{1}{\kappa}
	\|d^k\|^2$ whenever $k\in\N$ is sufficiently large. Then by
	Proposition~\ref{Maratos} we get \eqref{backtrackinghold}, which
	readily verifies the claimed assertion in case (i).  Assuming now
	the conditions in (ii) and using Claim 1, we easily get that the
	sequence $\{d^k\}$ converges to $0$ as $k\to\infty$, and that
	$x^k+d^k \to \bar{x}$ as $k \to \infty$. Employing the uniform
	second-order growth condition for tilt-stable minimizers from
	\cite[Theorem~3.2]{MorduNghia} gives us a neighborhood $U$ of
	$\bar{x}$ such that \begin{equation*}
	\varphi(x)\ge\varphi(u)+\langle\nabla\varphi(u),x-u\rangle+\frac{1}{2\kappa}\|x-u\|^2\quad\text{for
		all }\;x,u\in U, \end{equation*} which implies the inequality
	\eqref{growthdk}. Then by Proposition~\ref{Maratos} we get
	\eqref{backtrackinghold},
	and hence complete verifying the the claimed assertion in this case.\\[1ex]
	{\bf Claim~3:} {\em The conclusions of the theorem about the
		Q-superlinear convergence of the sequences $\{x^k\}$,
		$\{\ph(x^k)\}$, and $\{\nabla\ph(x^k)\}$ hold provided that  either
		condition {\rm(i)} or {\rm(ii)} of this theorem holds}. To justify
	this assertion, we get from Claim~2 that $\tau_k=1$ for all
	sufficiently large $k\in\N$, and thus Algorithm~\ref{LS} eventually
	becomes Algorithm \ref{NM}. Hence the claimed convergence results
	follow from \cite[Theorems~5.7 and 5.12]{BorisKhanhPhat}.  This
	verifies the assertions of Claim~3 and completes the proof of the
	theorem.\end{proof}

\section{ {GDNM for Problems of Quadratic Composite Optimization}}\label{sec:dampnon}

In this section we study a special class of constrained optimization problem given in the form
\begin{eqnarray}\label{QP0}
\text{minimize }\;\varphi(x):=f(x)+g(x),\quad x\in\R^n,
\end{eqnarray}
where $f\colon\R^n\to\R$ is a convex and smooth function, while $g\colon\R^n\to\oR$ is a convex and extended-real-valued one. This class is known as problems of {\em convex composite optimization}. Observe that, although \eqref{QP0} is written in the unconstrained format, it includes constrained optimization problems with the effective constraints $x\in\dom g$.

In what follows we pay the main attention to a subclass of \eqref{QP0} described by
\begin{eqnarray}\label{QP}
\text{minimize }\;\varphi(x):=\frac{1}{2}\langle Ax,x\rangle+\langle b,x\rangle+\alpha+g(x),\quad x\in\R^n,
\end{eqnarray}
where $A\in\R^{n\times n}$ is a positive-semidefinite matrix, $b\in\R^n$, and $\alpha\in\R$. That is, \eqref{QP} is a subclass of \eqref{QP0} with $f$ being a quadratic function. This allows us to label \eqref{QP} as problems of {\em quadratic composite optimization}.

Note that problems of type \eqref{QP} frequently appear, e.g., in practical models of machine learning and statistics. In particular, various {\em Lasso problems} considered in Section~\ref{Lassosec} can be written in form \eqref{QP}. Here are some other important classes of problems arising in practical modeling, which are reduced to \eqref{QP}.

\begin{Example}[\bf support vector machine problems]\rm Given the training data $(x_i,y_i)$, $i=1,\ldots,m$, where $x_i\in\R^n$ are the observations and $y_i\in\{-1,1\}$ are the labels, the \textit{support vector classification} is a problem of finding a hyperplane $y=\langle w,x\rangle+b$ such that the data with different labels can be separated by the hyperplane. One of the most popular {\em support vector machine} models \cite{HCL} is the regularized penalty model
	\begin{eqnarray}\label{SVM}
	\text{minimize }\;\varphi(w,b):=\frac{1}{2}\|w\|^2+C\sum_{i=1}^{m}\xi(w,x_i,y_i,b)\;\mbox{ with }\;w\in\R^n\;\mbox{ and }\;b\in\R,
	\end{eqnarray}
	where $C>0$ is a penalty parameter, and where $\xi\colon\R^n\times\R^n\times\R\times\R\to\R$ is called a {\em loss function}. Typical loss functions are given as follows:
	\begin{itemize}
		\item[\bf(i)] \textit{L1-loss}, or \textit{$\ell_1$ hinge loss}: $\xi(w,x_i,y_i,b)=\max\{1-y_i(\langle w,x_i\rangle+b),0\}$.
		\item[\bf(ii)] \textit{L2-loss}, or \textit{squared hinge loss}: $\xi(w,x_i,y_i,b)=\max\{1-y_i(\langle w,x_i\rangle+b),0\}^2$.
		\item[\bf(iii)] \textit{Logistic loss}: $\xi(w,x_i,y_i,b)=\log(1+e^{-y_i(\langle w,x_i\rangle+b)})$.
	\end{itemize}
\end{Example}

\begin{Example}[\bf convex clustering problems] \rm Let $A:=[a_1,\ldots,a_n]\in\R^{d\times n}$ be a given matrix with $n$ observations and $d$ features. The {\em convex clustering} model \cite{BMPS} is described in the form of quadratic composite optimization \eqref{QP} by:
	\begin{eqnarray}\label{clustering}
	\text{minimize }\; \frac{1}{2}\sum_{i=1}^{n}\|x_i-a_i\|^2+\gamma\sum_{i<j}\|x_i-x_j\|_p,\quad X\in\R^{d\times n},
	\end{eqnarray}
	where $\gamma>0$ is a tuning parameter, and where $\|\cdot\|_p$ denotes the $p$-norm. Typically the norm parameter $p$ is chosen as $1,2,\infty$.
\end{Example}

\begin{Example}[\bf constrained quadratic optimization] \rm Consider problem \eqref{QP}, where $g$ is the indicator function of a nonempty, closed, and convex set $\Omega$. Then \eqref{QP} is known as a {\em constrained quadratic optimization problem}. Some of the typical constraints are given by:
	\begin{itemize}
		\item[\bf(i)] \textit{Box constrained set}: $\Omega=\text{Box}[l,u]:=\big\{x\in\R^n\;\big|\;l\le x_i\le u,\;i=1,\ldots,n\big\}$.
		\item[\bf(ii)] \textit{Half-space}: $\Omega=\big\{x\in\R^n\;\big|\;\langle a,x\rangle\le\alpha\big\}$, where $a\in\R^n\setminus\{0\}$ and $\alpha\in\R$.
		\item[\bf(iii)] \textit{Affine set}: $\Omega=\big\{x\in \R^n\;\big|\;Ax=b\big\}$, where $A$ is an $m\times n$ matrix and $b\in\R^m$.
	\end{itemize}
\end{Example}

\begin{Remark}[\bf subproblems for other methods] \rm Quadratic composite optimization problems of type \eqref{QP} not only cover optimization models in machine learning, statistics, and other applied areas, but also arise as {\em subproblems} for some efficient algorithms including {\em sequential quadratic programming methods} (SQP) \cite{Bonnans,Solo14}, {\em augmented Lagrangian methods} \cite{hangborisebrahim,He69,lsk,Po69,Roc74,roc}, {\em proximal Newton methods} \cite{lss,myzz}, etc.
\end{Remark}

To develop now a globally convergent damped Newton method for solving quadratic composite optimization problems of type \eqref{QP}, we employ the basic machinery of variational analysis, which allows us to reduce \eqref{QP} to unconstrained problems with ${\cal C}^{1,1}$ objectives. Following \cite{Rockafellar98}, recall the corresponding notions used in our subsequent developments.

\begin{Definition}[\bf Moreau envelopes and proximal mappings]\label{def:moreau} Given an extended-real-valued, proper, l.s.c.\ function $\varphi\colon\R^n\to\oR$ and given a parameter value $\gamma>0$, the {\sc Moreau envelope} $e_\gamma\varphi$ and the {\sc proximal mapping} $\textit{\rm Prox}_{\gamma\varphi}$ are defined by
	\begin{equation}\label{Moreau}
	e_\gamma\varphi(x):=\inf\left\{\varphi(y)+\frac{1}{2\gamma}\|y-x\|^2\;\Big|\;y\in\R^n\right\},
	\end{equation}
	\begin{equation}\label{Prox}
	\textit{\rm Prox}_{\gamma\varphi}(x):={\rm argmin}\left\{\varphi(y)+\frac{1}{2\gamma}\|y-x\|^2\;\Big|\;y\in\R^n\right\}.
	\end{equation}
	If $\gamma=1$, we use the notation $e\varphi(x)$ and $\text{\rm Prox}_{\varphi}(x)$ in \eqref{Moreau} and \eqref{Prox}, respectively.
\end{Definition}

Both Moreau envelopes and proximal mappings have been well recognized in variational analysis and optimization as efficient tools of regularization and approximation of nonsmooth functions. Prior to
establishing the main convergence results of this section, we present several lemmas of their own interest. The first one, taken from \cite[Proposition~12.30]{Bauschke}, lists those properties of Moreau
envelopes and proximal mappings for convex extended-real-valued functions that are needed to derive the main results obtained below.

\begin{Lemma}[\bf Moreau envelopes and proximal mappings for convex functions]\label{rela} Let $\varphi\colon\R^n\to\oR$ be a proper, l.s.c., convex function. Then the following assertions hold for all $\gamma>0$:
	\begin{itemize}
		\item[\bf(i)] The Moreau envelope $e_\gamma\varphi$ is of class of continuously differentiable functions, and its gradient is Lipschitz continuous with modulus $1/\gamma$ on $\R^n$.
		\item[\bf(ii)] The proximal mapping $\text{\rm Prox}_{\gamma\ph}$ is single-valued, monotone, and nonexpansive, i.e., it is Lipschitz continuous with modulus $1$ on $\R^n$.
		\item[\bf(iii)] The gradient of $e_\gamma\ph$ is calculated by
		\begin{equation}\label{gradientMoreau}
		\nabla e_\gamma\varphi(x)=\frac{1}{\gamma}\Big(x-\text{\rm Prox}_{\gamma\varphi}(x)\Big)=\big(\gamma I+(\partial\varphi)^{-1}\big)^{-1}(x)\;\mbox{ for all }\;x\in \R^n.
		\end{equation}
	\end{itemize}
\end{Lemma}

The results of Lemma~\ref{rela} allow us to pass from nonsmooth convex optimization problems of type \eqref{QP} with extended-valued objectives (i.e., including constraints) to an unconstrained
${\cal C}^{1,1}$ problem given in form \eqref{opproblem}. Note that such an approach has been used in \cite{BorisKhanhPhat,BorisEbrahim} to design locally convergent pure Newton algorithms for optimization problems and subgradient inclusions associated with prox-regular functions \cite{Rockafellar98}. However, now we go further from the numerical viewpoint. Exploiting the quadratic composite structure of problems \eqref{QP} and their specifications leads us to the design and justification of a new {\em globally} convergent algorithm with {\em constructive} calculations of its parameters via the problem data.\vspace*{0.05in}

To proceed, let $\gamma>0$ be such that the matrix $I-\gamma A$ is positive-definite. Denoting $Q:=(I-\gamma A)^{-1}$, $c:= \gamma Qb$, and $P:=Q-I$, define the unconstrained optimization problem by
\begin{eqnarray}\label{FBOQPptimization}
\text{minimize}&\text{}&\psi(u):=\frac{1}{2}\langle Pu,u\rangle+\langle c,u\rangle+\gamma e_\gamma g(u)\;\mbox{ subject to }\;u\in\R^n.
\end{eqnarray}

The following lemma reveals some important properties of the optimization problem \eqref{FBOQPptimization}.

\begin{Lemma}[\bf quadratic composite problems via proximal mappings]\label{psiform} Let $\psi$ be given in \eqref{FBOQPptimization}. Then $\psi$ is a continuously differentiable function represented by
	\begin{equation}\label{psiexpress}
	\psi(u):=\frac{1}{2}\langle Pu,u\rangle+\langle c,u\rangle+\gamma g\big(\text{\rm Prox}_{\gamma g}(u)\big)+\frac{1}{2}\|u-\text{\rm Prox}_{\gamma g}(u)\|^2.
	\end{equation}
	Moreover, the mapping $\nabla\psi$ is Lipschitz continuous on $\R^n$ with modulus $\ell:=1+\|Q\|$, and
	\begin{equation}\label{nablapsiexpress}
	\nabla\psi(u)=Qu-\text{\rm Prox}_{\gamma g}(u)+c.
	\end{equation}
	If in addition $A$ is positive-definite, then $\psi$ is strongly convex with modulus $\lambda_{\text{\rm min}}(P)>0$.
\end{Lemma}
\begin{proof} Due to the convexity of $g$ and Lemma~\ref{rela}, the Moreau envelope $e_\gamma g$ is continuously differentiable and the proximal mapping $\text{\rm Prox}_{\gamma g}$ is nonexpansive on $\R^n$. Thus the function $\psi$ is continuously differentiable as well. The representations in \eqref{psiexpress} and \eqref{nablapsiexpress} follow from the definition of $\psi$ and formula \eqref{gradientMoreau}. Furthermore, for any $u_1,u_2\in\R^n$ we have
	$$
	\|\nabla\psi(u_1)-\nabla\psi(u_2)\|=\|Qu_1-Qu_2-\text{\rm Prox}_{\gamma g}(u_1)+\text{\rm Prox}_{\gamma g}(u_2)\|\le(1+\|Q\|)\|u_1-u_2\|=\ell\|u_1-u_2\|,
	$$
	which justifies the global Lipschitz continuity of $\psi$ on $\R^n$ with the uniform constant $\ell$ defined above. Suppose further that $A$ is positive-definite. Combining this with the
	positive-definiteness of $I-~\gamma A$ readily yields the positive-definiteness of $P$. Thus $\psi$ in \eqref{FBOQPptimization} is strongly convex on $\R^n$ with modulus $\lambda_{\text{\rm min}}(P)>~0$.
\end{proof}

Next we establishes the relationship between the optimization problems \eqref{QP} and \eqref{FBOQPptimization}.

\begin{Lemma}[\bf reduction of quadratic composite problems to ${\cal C}^{1,1}$ optimization]\label{characterizeQP} Consider the optimization problems \eqref{QP} and \eqref{FBOQPptimization}.
	The following are equivalent:
	\begin{itemize}
		\item[\bf(i)] The vector $\bar{x}$ is an optimal solution to \eqref{QP}.
		\item[\bf(ii)] The vector $\bar{x}=Q\bar{u}+c$, where $\bar{u}$ is an optimal solution to \eqref{FBOQPptimization}.
	\end{itemize}
\end{Lemma}
\begin{proof} Using \cite[Theorem~26.2]{Bauschke} and the expression $\nabla f(x):=Ax+b$ for all $x\in\R^n$ tells us that the optimal solution to \eqref{QP} is fully characterized by the equation
	\begin{equation}\label{FermatQP}
	x-\text{\rm Prox}_{\gamma g}\big(x-\gamma(Ax+b)\big)=0.
	\end{equation}
	For each $x\in\R^n$ denote $u:=x-\gamma(Ax+b)=(I-\gamma A)x-\gamma b$ and observe by the positive-definiteness of the matrix $I-\gamma A$ that \eqref{FermatQP} is equivalent to
	\begin{equation}
	\begin{cases}
	Qu-\text{\rm Prox}_{\gamma g}(u)+c=0,\\
	x=Qu+c,
	\end{cases}
	\end{equation}
	where $Q:=(I-\gamma A)^{-1}$ and $c:=\gamma Qb$. The positive-definiteness of $I-\gamma A$ and the positive-semidefiniteness of $A$ imply that $P=Q-I$ is positive-semidefinite. Furthermore, the
	convexity
	of $g$ and Lemma~\ref{psiform} ensure that $e_\gamma g$ is continuously differentiable on $\R^n$, and that $\bar{u}$ is a solution to \eqref{FBOQPptimization} if and only if we have
	$$
	0=\nabla\psi(\bar{u})=P\bar{u}+c+\gamma\nabla e_\gamma g(\bar{u})=Q\bar{u}-\text{\rm Prox}_{\gamma g}(\bar{u})+c.
	$$
	This verifies the equivalence between (i) and (ii) as stated in the lemma.
\end{proof}

The last lemma of this section provides the representation of the second-order subdifferential of the cost function $\psi$ in the reduced problem \eqref{FBOQPptimization} via the second-order subdifferential of
the
given regularizer $g$ in the original one \eqref{QP}.

\begin{Lemma}[\bf second-order subdifferentials of reduced cost
	functions]\label{calculatepsi} Let $\psi\colon\R^n\to\R$ be taken
	from \eqref{FBOQPptimization}. Then for each $u\in\R^n$ and
	$w\in\R^n$ we have the relationship $$
	z\in\partial^2\psi(u)(w)\iff\frac{1}{\gamma}(z-Pw)\in\partial^2g\Big(\text{\rm
		Prox}_{\gamma g}(u),\frac{1}{\gamma}\big(u-\text{\rm Prox}_{\gamma
		g}(u)\big)\Big)\big(Qw-z\big). $$ \end{Lemma} \begin{proof} Using
	the second-order subdifferential sum rule from
	\cite[Proposition~1.121]{Mordukhovich06} gives us $$
	\partial^2\psi(u)(w)=Pw+\gamma\partial^2e_\gamma
	g\Big(u,\frac{1}{\gamma}\big(\nabla\psi(u)-Pu-c\big)\Big)(w). $$
	This we have that $z\in\partial^2\psi(u)(w)$ if and only if $$
	\frac{1}{\gamma}(z-Pw)\in\partial^2 e_\gamma
	g\Big(u,\frac{1}{\gamma}\big(\nabla\psi(u)-Pu-c\big)\Big)(w). $$ Due
	to \cite[Lemma~6.4]{BorisKhanhPhat}, the latter is equivalent to
	\begin{equation}\label{secondofg}
	\frac{1}{\gamma}(z-Pw)\in\partial^2g\Big(u-\nabla\psi(u)+Pu+c,\frac{1}{\gamma}\big(\nabla\psi(u)-Pu-c\big)\Big)(w-z+Pw).
	\end{equation} Furthermore, we have the equalities
	\begin{equation}\label{component1}
 {u-\nabla\psi(u)+Pu+c=u-\gamma\nabla e_\gamma g(u)=\text{\rm
		Prox}_{\gamma g}(u),} \end{equation}
	\begin{equation}\label{component2}
	\frac{1}{\gamma}(\nabla\psi(u)-Pu-c)=\frac{1}{\gamma}\big(u-\text{\rm
		Prox}_{\gamma g}(u)\big). \end{equation} Combining \eqref{secondofg}
	with \eqref{component1} and \eqref{component2} completes the proof
	of the lemma.\end{proof}

Now we are in a position to design the aforementioned generalized damped Newton-type algorithm to solve problems \eqref{QP} of quadratic composite optimization.

\begin{Algorithm}[\bf generalized damped Newton algorithm for problems of quadratic composite optimization]\label{LSQP}\rm \text{}
	
	\medskip\noindent
	\textbf{Input:} $A\in\R^{n\times n}$, $b\in\R^n$, $g$, $\sigma\in\left(0,\frac{1}{2}\right)$, and $\beta\in(0,1)$. Do the following:\\[1ex]
	{\bf Step~0:}  Choose $\gamma>0$ such that $I-\gamma A$ is positive-definite, calculate $Q:=(I-\gamma A)^{-1}$, $c:=\gamma Qb$, $P:=Q-I$, define the function $\psi$ as in \eqref{psiexpress},
	choose a starting point $u^0\in\R^n$, and set $k:=0$.\\[1ex]
	{\bf Step~1:} If $\nabla\psi(u^k)=0$, then stop. Otherwise, set $v^k:=\text{\rm Prox}_\gamma g(u^k)$.\\[1ex]
	{\bf Step~2}: Find $d^k\in\R^n$ such that
	\begin{equation}\label{prox-dir}
	\frac{1}{\gamma}\big(-\nabla\psi(u^k)-Pd^k\big)\in\partial^2g\Big(v^k,\frac{1}{\gamma}(u^k-v^k)\Big)\big(Qd^k+\nabla\psi(u^k)\big).
	\end{equation} {\bf Step~3:} Set $\tau_k=1$. Until \textit{ Armijo's
		inequality} $$ \psi(u^k+\tau_kd^k)\le
	\psi(u^k)+\sigma\tau_k\langle\nabla\psi(u^k),d^k\rangle $$
	is satisfied, set $\tau_k:=\beta\tau_k$.\\[1ex]
	{\bf Step~4:} Compute $u^{k+1}$ by
	\begin{equation*}
	u^{k+1}:=u^k+\tau_kd^k,\quad k=0,1,\ldots.
	\end{equation*}
	{\bf Step 5:} Increase $k$ by $1$ and go to Step~1. \\[1ex]
	\textbf{Output}: $x^k:=Qu^k +c$.
\end{Algorithm}

Note that the definitions of the second-order subdifferential \eqref{2nd} and the limiting coderivative \eqref{lim-cod} allow us to rewrite the implicit inclusion \eqref{prox-dir} for $d^k$ in the
{\em explicit form}
\begin{equation}\label{prox-dir1}
\Big(\frac{1}{\gamma}\big(-\nabla\psi(u^k)- Pd^k\big),-Qd^k-\nabla\psi(x^k)\Big)\in N_{\text{gph}\,\partial g}\Big(v^k,\frac{1}{\gamma}(u^k-v^k)\Big).
\end{equation}
Explicit expressions for the sequences $\{v^k\}$ and $\{d^k\}$ in Algorithm \ref{LSQP} depend on given structures of the regularizers $g$, which are efficiently specified in applied models of machine learning and statistics; see, e.g., the above discussions and those presented in Section~\ref{Lassosec}.

\begin{Remark}[\bf stopping criterion]\label{rem:stop} \rm Note that $\bar{x}$ is a solution to \eqref{QP} if and only if this point satisfies the stationary equation \eqref{FermatQP}. In order to approximate the solution $\bar{x}$, we choose the {\em termination/stopping criterion}
	\begin{equation}\label{stop}
	\big\|x-\text{\rm Prox}_{\gamma g}\big(x-\gamma(Ax+b)\big)\big\|\le\epsilon
	\end{equation}
	with a given {\em tolerance} parameter $\epsilon>0$. The stopping criterion \eqref{stop} is clearly equivalent to the condition  $\|\nabla\psi(u)\|\le\epsilon$, where $u:=x-\gamma(Ax+b)=Q^{-1}(x-c)$, and where $\psi$ is defined in \eqref{psiexpress}. Therefore, in practice the stopping criterion of Step~2 of Algorithm~\ref{LSQP} can be replaced by the simpler one $\|\nabla\psi(u^k)\|\le~\epsilon$.
\end{Remark}

To proceed with establishing conditions for global convergence of Algorithm~\ref{LSQP}, we need to employ yet another notion of generalized second-order differentiability taken from \cite[Chapter~13]{Rockafellar98}. First recall that a mapping $h\colon\R^n\to\R^m$ is {\em semidifferentiable} at $\bar{x}$ if there exists a continuous and positively homogeneous operator $H\colon\R^n\to\R^m$ such that
$$
h(x)=h(\bar{x})+H\big(x-\bar{x}\big)+o(\|x-\bar{x}\|)\quad\text{for all }\;x\;\text{ near }\;\bar{x}.
$$
Given $\ph\colon\R^n\to\oR$ with $\ox\in\dom\ph$, consider the family of second-order finite differences
\begin{equation*}
\Delta^2_\tau\varphi(\bar{x},v)(u):=\frac{\varphi(\bar{x}+\tau u)-\varphi(\bar{x})-\tau\langle v,u\rangle}{\frac{1}{2}\tau^2}
\end{equation*}
and define the {\em second subderivative} of $\varphi$ at $\ox$ for $v\in\R^n$ and $w\in\R^n$ by
\begin{equation*}
d^2\varphi(\ox,v)(w):=\liminf_{\tau\downarrow 0\atop u\to w}\Delta^2_\tau\varphi(\ox,v)(u).
\end{equation*}
Then $\ph$ is said to be {\em twice epi-differentiable} at $\ox$ for $v$ if for every $w\in\R^n$ and every choice of $\tau_k\downarrow 0$ there exists a sequence $w^k\to w$ such that
\begin{equation*}
\frac{\varphi(\ox+\tau_k w^k)-\varphi(\ox)-\tau_k\langle v,w^k\rangle}{\frac{1}{2}\tau_k^2}\to d^2\varphi(\ox,v)(w)\;\mbox{ as }\;k\to\infty.
\end{equation*}
Twice epi-differentiability has been recognized as an important property in second-order variational analysis with numerous applications to optimization; see the aforementioned monograph by Rockafellar and Wets and the recent papers \cite{mms,mms1,ms} developing a systematic approach to verify epi-differentiability via {\em parabolic regularity}, which is a major second-order property of sets and extended-real-valued functions.\vspace*{0.05in}

The next theorem provides verifiable conditions on the matrix $A$ and the function $g$ to run the GDNM Algorithm~\ref{LSQP} for solving quadratic composite optimization problems \eqref{QP}.

\begin{Theorem}[\bf convergence rate of GDNM for problems of quadratic composite optimization]\label{solvingQP} Considering the optimization problem \eqref{QP}, suppose that $A$ is positive-definite. Then the following assertions hold:
	\begin{itemize}
		\item[\bf(i)] Algorithm~{\rm\ref{LSQP}} is well-defined, and the sequence of its iterates $\{u^k\}$ globally R-linearly converges  to some $\bar{u}$ as $k\to\infty$.
		\item[\bf(ii)] The vector $\bar{x}:=Q\bar{u}+c$ is a tilt-stable local minimizer of the cost function $\varphi$ in \eqref{QP}, and $\ox$ is the unique solution to \eqref{QP}.
	\end{itemize}
	Furthermore, the rate of convergence of $\{u^k\}$ is at least Q-superlinear if the mapping $\partial g$ is semismooth$^*$ at  {$(\ox,\bar{v})$ with $\bar{v}:=-A\ox-b$}, and if one of two following conditions fulfills:
	\begin{itemize}
		\item[\bf(a)] $\sigma\in\left(0,1/(2\ell\kappa)\right)$, where $\ell:=1+\|Q\|$ and $\kappa:=1/\lambda_{\text{\rm min}}(P)$.
		\item[\bf(b)]  {$g$ is twice epi-differentiable at $\ox$ for $\bar{v}$.}
	\end{itemize}
\end{Theorem}
\begin{proof} Lemma~\ref{psiform} and Lemma~\ref{calculatepsi} tell us that applying Algorithm~\ref{LSQP} to solve the quadratic composite optimization problem \eqref{QP} is equivalent to applying
	Algorithm~\ref{LS} to solving the ${\cal C}^{1,1}$ optimization problem \eqref{FBOQPptimization}. We split the proof of the theorem into the following claims:\\[1ex]
	{\bf Claim 1:} {\em The function $\psi$ from \eqref{FBOQPptimization} satisfies Assumptions~{\rm\ref{PDofHessian}} and {\rm\ref{boundedoflevelset}}.} Indeed, Lemma~\ref{psiform} ensures that $\psi$ is
	strongly convex with modulus $\lambda_{\text{\rm min}}(P)>0$, and thus Assumption~\ref{PDofHessian} holds due to \cite[Theorem~5.1]{ChieuChuongYaoYen}. Moreover, the strong convexity of $\psi$ clearly implies
	that for an arbitrary vector $u^0\in\R^n$ the level set
	$$
	\Omega:=\big\{u\in\R^n\;\big|\;\psi(u)\le\psi(u^0)\big\}
	$$
	is bounded, and so Assumption~\ref{boundedoflevelset} holds for the function $\psi$.\\[1ex]
	{\bf Claim 2:} {\em Both statements {\rm(i)} and {\rm(ii)} of the
		theorem are satisfied.} To proceed, we employ Claim~1 together with
	Theorems~\ref{sufconver} and \ref{linearcon} to conclude that
	 Algorithm~\ref{LSQP} is well-defined, and that the sequence of its
 {iterates $\{u^k\}$ globally R-linearly converges to some $\bar{u}$}
	as $k\to\infty$. Then Lemma~\ref{characterizeQP} tells us that the
	vector $\bar{x}=Q\bar{u}+c$ is a solution to \eqref{QP}. The
	uniqueness and tilt stability of $\bar{x}$ follow immediately from
	the strong
	convexity of $\varphi$.\\[1ex]
	{\bf Claim 3:} {\em The convergence rate of the sequence $\{u^k\}$ is at least Q-superlinear provided that $\partial g$ is semismooth$^*$ at all the points on its graph and that either one of the two conditions {\rm(a)} and {\rm(b)} is satisfied.} Indeed, suppose that {$\partial g$ is semismooth$^*$ at $(\bar{x},\bar{v})$. Then we deduce from \cite[Proposition~6]{fgh} that $\text{\rm Prox}_{\gamma g}$ is semismooth$^*$ at $\bar{x}-\gamma(A\bar{x}+b)$.  Thus we obtain that the mapping $\nabla\psi(u)=Qu-\text{\rm Prox}_{\gamma g}(u)+c$ is semismooth$^*$ at $\bar{u}$ by employing \cite[Proposition~3.6]{Helmut}.}
	
	Assume now that condition (a) is satisfied. Then Lemma~\ref{psiform} tells us that $\ell$ is a Lipschitz constant of $\nabla\psi$ around $\bar{u}$, and that $\bar{u}$ is a tilt-stable local minimizer of
	$\psi$ with modulus $\kappa$. Thus the claimed assertion follows in case (a) directly from Theorem~\ref{superlinearLS}.
	
	Assuming by (b) that {$g$ is twice epi-differentiable at $\bar{x}$ for $\bar{v}$, we deduce from \cite[Theorem~13.40]{Rockafellar98} that  the subgradient mapping $\partial g$ is proto-differentiable of at $(\bar{x},\bar{v})$.	Using \cite[Corollary~8]{fgh}, we conclude that $\text{\rm Prox}_{\gamma g}$ is directionally differentiable at $\bar{x}-\gamma(A\bar{x}+b)$, which yields in turn the directional differentiability of $\nabla\psi$ at $\bar{u}$.  Finally, Theorem~\ref{superlinearLS} allows us to conclude that the sequence $\{u^k\}$ Q-superlinearly converges to $\bar{u}$ as $k\to\infty$.} \end{proof}

It is highly desired to obtain a counterpart of Theorem~\ref{solvingQP} on global convergence of Algorithm~\ref{LSQP} under merely {\em positive-semidefiniteness} of the matrix $A$. However, we cannot do it at this stage of developments since the function $\psi$ from \eqref{FBOQPptimization} may not satisfy Assumption~\ref{PDofHessian}. A natural idea to overcome such a challenge is to {\em regularize} the original problem via approximating it by a sequence of well-behaved problems. Perhaps the simplest way to realize this idea is to employ the classical {\em Tikhonov regularization}. To this end, consider in the setting of Lemma~\ref{characterizeQP} the following family of optimization problem depending on the parameter $\epsilon>0$:
\begin{eqnarray}\label{FBregularization}
\text{minimize}&\text{}&\psi_\epsilon(u):=\frac{1}{2}\langle P_\epsilon u,u\rangle+\langle c,u\rangle+\gamma e_\gamma g(u)\;\mbox{ subject to }\;u\in\R^n,
\end{eqnarray}
where $P_\epsilon:=P+\epsilon I$. The next proposition discusses the relationship between \eqref{FBregularization} and \eqref{QP}.

\begin{Proposition}[\bf Tikhonov regularization]\label{Tikhonov} Assume that \eqref{QP} has a solution, and for each $\epsilon>0$ consider the optimization problem \eqref{FBregularization}. If $\bar{u}(\epsilon)$ is a solution to \eqref{FBregularization}, then we have the following assertions:
	\begin{itemize}
		\item[\bf(i)] The limit $\displaystyle\bar{u}:=\lim_{\epsilon\to 0}\bar{u}(\epsilon)$ exists being a solution to \eqref{FBOQPptimization}.
		\item[\bf(ii)] The vector $\bar{x}:=Q\bar{u}+c$ is a solution to \eqref{QP}.
	\end{itemize}
\end{Proposition}
\begin{proof} Observe that the optimization problem \eqref{FBOQPptimization} is equivalent to the {\em variational inequality problem} VI($\R^n,F$) written as: find a vector $u\in\R^n$ such that
	$$
 {\langle F(u),z-u\rangle\ge 0\;\mbox{ for all }\;z\in\R^n,}
	$$
	where $F:=\nabla\psi$. Since $\bar{u}(\epsilon)$ is a solution to \eqref{FBregularization}, we get that the family of approximate solutions $\{\bar{u}(\epsilon)\;|\;\epsilon>0\}$ is the {\rm Tikhonov trajectory} of VI($\R^n,F$); see, e.g., \cite[Equation (12.2.2)]{JPang}. It follows from the convexity of $\psi$ that $\nabla\psi\colon\R^n\to\R^n$ is a monotone operator. Since the optimization problem \eqref{QP} has a solution, the set of solutions for VI($\R^n,F$) is nonempty by Lemma~\ref{characterizeQP}. Using \cite[Theorem~12.2.3]{JPang}, we have that the limit $\displaystyle\bar{u}=\lim_{\epsilon\to 0}u(\epsilon)$ exists and solves \eqref{FBOQPptimization}. Finally, assertion (ii) follows immediately from {Lemma~\ref{characterizeQP}}.\end{proof}

\begin{Remark}[\bf generalized Newton algorithm based on Tikhonov regularization] \rm Proposition~\ref{Tikhonov} provides a precise relationship between solutions to \eqref{QP} and solutions to \eqref{FBregularization}. This plays a crucial role in solving \eqref{QP} without assuming the positive-definiteness of $A$. Moreover, Proposition~\ref{Tikhonov} motivates us to develop a generalized Newton-type algorithm based on the Tikhonov regularization to solve the class of optimization problems \eqref{QP}, where $A$ is merely positive-semidefinite. We will pursue this issue in our future research.
\end{Remark}\vspace*{-0.2in}

\section{ {Solving Lasso Problems and Numerical Experiments}}\label{Lassosec}

  This section is devoted to specifying our generalized damped Newton method (GDNM) developed in Section \ref{sec:dampnon} for the basic class of {\em Lasso problems}, where Lasso stands for the {\em Least Absolute Shrinkage and Selection Operator.} Using on the obtained specification of GDNM to solve Lasso problems, we conduct numerical experiments by using our algorithm and the compare the computations with the performances of some major first-order and second-order algorithms applied to solving this class of problems of composite quadratic optimization.

 The basic Lasso problem, known also as the $\ell^1$-{\em regularized least square optimization problem}, was introduced by Tibshirani \cite{Tibshirani} and then has been largely investigated and applied to various issues in statistics, machine learning, image processing, etc. This problem is formulated as:
\begin{eqnarray}\label{Lasso}
\text{minimize }\;\varphi(x):=\frac{1}{2}\|Ax-b\|_2^2+\mu\|x\|_1\quad\text{ subject to }\;x\in\R^n,
\end{eqnarray}
where $A$ is an $m\times n$ matrix, $\mu>0$, $b\in\R^m$, and where  
$$
\|x\|_1:= \sum_{i=1}^{n}|x_i|,\quad\|x\|_2:= \left(\sum_{i=1}^{n} |x_i|^2 \right)^{1/2} \quad \text{for all }\; x=(x_1,x_2,\ldots,x_n).
$$
\color{black} There are other classes of Lasso problems modeled in the quadratic composite form
\begin{eqnarray}\label{leastsquare}
\text{minimize }\;\varphi(x):=\frac{1}{2}\|Ax-b\|_2^2+g(x),\quad x\in\R^n,
\end{eqnarray}
where $A$ is an $m\times n$ matrix, $b\in\R^m$ and $g\colon\R^n\to\overline{\R}$ is a given {\rm regularizer}. More specifically, let us list several well-recognized versions of \eqref{leastsquare} in addition to \eqref{Lasso}:
\begin{itemize}
	\item[\bf(i)] The {\em elastic net regularized} problem, or the {\em Lasso elastic net} problem \cite{HZ05} with
	$$
	g(x):=\mu_1\|x\|_1+\mu_2\|x\|_2^2,
	$$
	where $\mu_1$ and $\mu_2$ are given positive parameters.
	\item[\bf(ii)] The {\em clustered Lasso} problem \cite{she} with
	$$
	g(x):=\mu_1\|x\|_1+\mu_2\sum_{1\le i\le j\le n}|x_i-x_j|,
	$$
	where $\mu_1$ and $\mu_2$ are given positive parameters.
	\item[\bf(iii)] The {\em fused regularized} problem, or the {\em fused Lasso} problem \cite{TSRZK} with
	$$
	g(x):=\mu_1\|x\|_1+\mu_2\|Bx\|_1,
	$$
	where $\mu_1$ and $\mu_2$ are given positive numbers, and where $B$ is the $(n-1)\times n$ matrix
	$$
	Bx:=[x_1-x_2,x_2-x_3,\ldots,x_{n-1}-x_n]\quad\text{for all }\;x\in\R^n.
	$$
\end{itemize}\vspace*{0.05in}

Although the developed Algorithm~\ref{LSQP} allows us to efficiently solve all these Lasso problems, we concentrate here on numerical results for the basic one \eqref{Lasso}. It is easy to see that the Lasso problem \eqref{Lasso} belongs to the quadratic composite class \eqref{QP}. Indeed, we represent \eqref{Lasso} as minimizing the nonsmooth convex function $\varphi(x):=f(x)+g(x)$, where
\begin{equation}\label{fg}
f(x):=\frac{1}{2}\langle\bar{A}x,x\rangle+\langle\bar{b},x\rangle+\bar{\alpha}\quad\text{and }\;g(x):=\mu\|x\|_1
\end{equation}
with $\bar{A}:=A^*A$, $\bar{b}:=-A^*b$, and $\bar{\alpha}:=\frac{1}{2}\|b\|^2$, and where the matrix $\bar{A}=A^*A$ is positive-semidefinite. Observe further that the Lasso problem \eqref{Lasso} always admits an optimal solution; see \cite{Tibshirani}. In order to apply Algorithm~\ref{LSQP} to solving problem \eqref{Lasso}, we begin with providing explicit calculations of the first-order and second-order subdifferentials of the regularizer $g(x)=\mu\|x\|_1$ together with the proximal mapping associated with this function.\vspace*{0.05in}

Using definition \eqref{Prox}, it is not hard to compute the proximal mapping of $g(x)=\mu\|x\|_1$ by
\begin{equation}\label{proxofg}
\big(\text{\rm Prox}_{\gamma g}(x)\big)_i=\begin{cases}
x_i-\mu\gamma&\text{if}\quad x_i>\mu\gamma,\\
0&\text{if}\quad-\mu\gamma\le x_i\le\mu\gamma,\\
x_i+\mu\gamma&\text{if}\quad x_i<-\mu\gamma.
\end{cases}
\end{equation}
Now we compute the first-order and second-order subdifferentials of this function.

\begin{Proposition}[\bf subdifferential calculations] For the regularizer $g(\cdot)= \mu\|\cdot\|_1$ in \eqref{Lasso} we have the subgradient mapping
	\begin{equation}\label{first-order}
	\partial g(x)=\left\{v\in\R^n\;\bigg|\;
	\begin{array}{@{}cc@{}}
	v_j=\text{\rm sgn}(x_j),\;x_j\ne 0,\\
	v_j\in[-\mu,\mu],\;x_j=0
	\end{array}\right\} \quad \text{whenever }\; x \in \R^n.
	\end{equation}
	Further, for each $(x, y)\in\text{\rm gph}\,\partial g$ and $v=(v_1,\ldots,v_n)\in\R^n$, the second-order subdifferential of $g$ is computed by the formula
	\begin{equation}\label{second-order}
	\partial^2 g(x,y)(v)=\Big\{w\in\R^n\;\Big|\;\Big(\frac{1}{\mu}w_i,-v_i\Big)\in G\Big(x_i,\frac{1}{\mu}y_i\Big),\;i=1,\ldots,n\Big\},
	\end{equation}
	where the mapping $G\colon\R^2\tto\R^2$ is defined by
	\begin{equation}\label{G}
	G(t,p):=\begin{cases}
	\{0\}\times\R&\text{\rm if}\quad t\ne 0,\;p\in\{-1,1\},\\
	\R\times\{0\}&\text{\rm if}\quad t=0,\;p\in(-1,1),\\
	(\R_{+}\times\R_{-})\cup(\{0\}\times\R)\cup(\R\times\{0\})&\text{\rm if}\quad t=0,\;p=-1,\\
	(\R_{-}\times\R_{+})\cup(\{0\}\times\R)\cup(\R\times\{0\})&\text{\rm if}\quad t=0,\;p=1,\\
	\emp&\text{\rm otherwise}.
	\end{cases}
	\end{equation}
\end{Proposition}
\begin{proof} {These computations follow from \cite[Proposition 8.1]{BorisKhanhPhat}}.\end{proof}

The next theorem provides an efficient condition on \eqref{Lasso} expressed entirely in terms of its given data to ensure a global superlinear convergence of Algorithm~\ref{LSQP} for solving \eqref{Lasso}.

\begin{Theorem}[\bf solving Lasso] \label{solveLasso} Considering the Lasso problem \eqref{Lasso}, suppose that the matrix $A^*A$ is positive-definite. Then we have:
	\begin{itemize}
		\item[\bf(i)] Algorithm~{\rm\ref{LSQP}} is well-defined, and the sequence of its iterates $\{y^k\}$ globally converges at least Q-superlinearly to $\bar{y}$ as $k\to\infty$.
		\item[\bf(ii)] The vector $\bar{x}:=Q\bar{y}+c$ is a unique solution to \eqref{Lasso} being a tilt-stable local minimizer of the cost function $\varphi$.
	\end{itemize}
\end{Theorem}
\begin{proof} It follows from \eqref{first-order} that the graph of $\partial g$ is the union of finitely many closed convex sets, and hence $\partial g$ is semismooth$^*$ on its graph. Furthermore, $g$ is proper, convex, and piecewise linear-quadratic on $\R^n$. Then \cite[Proposition~13.9]{Rockafellar98} ensures that $g$ is twice epi-differentiable on $\R^n$. Applying Theorem~\ref{solvingQP}, we arrive at all the conclusions of Theorem~\ref{solveLasso}.\end{proof}

To run Algorithm~\ref{LSQP}, we need to explicitly determine the sequences $\{v^k\}$ and $\{d^k\}$ generated by this algorithm. Using \eqref{proxofg}, \eqref{first-order}, and \eqref{second-order} gives us the following expressions for all the components $i=1,\ldots,n$ of these vectors:
$$
\left(v^k\right)_i=\begin{cases}
y_i-\mu\gamma&\text{if}\quad y_i>\mu\gamma,\\
0&\text{if}\quad-\mu\gamma\le y_i\le\mu\gamma,\\
y_i+\mu\gamma&\text{if}\quad y_i<-\mu\gamma,
\end{cases}
$$
$$
\begin{cases}
\big(Pd^k+\nabla\psi(y^k)\big)_i=0&\text{if}\quad\left(v^k\right)_i\ne 0,\\
\big(Qd^k+\nabla\psi(y^k)\big)_i=0&\text{if}\quad\left(v^k\right)_i=0.
\end{cases}
$$
\begin{Remark}[\bf Newton descent directions for Lasso] \rm Let us emphasize that the algorithm directions $d^k$ can be computed through solving a system of linear equations for each $k\in\N$. Indeed, for the sequence $\{d^k\}$ generated by Algorithm~\ref{LSQP}, denote by $P_i$ and $Q_i$ are the $i$-th rows of the matrices $P$ and $Q$, respectively. Define
	$$
	(X^k)_i:=\begin{cases}
	P_i&\text{if}\quad v_i\ne 0,\\
	Q_i&\text{if}\quad v_i=0.
	\end{cases}
	$$
	Then $d^k$ is a solution to the system of linear equations $X^k d=-\nabla\psi(y^k)$.
\end{Remark}

{Now we are in a position to conduct numerical experiments for solving the Lasso problem \eqref{Lasso} by using our generalized damped Newton method (GDNM) via Algorithm~\ref{LSQP}. The obtained calculations are compared with those obtained by implementing the following highly recognized first-order and second-order algorithms:
\begin{itemize}
	\item[\bf (i)] The {\em Alternating Direction Methods of
		Multipliers\footnote{\href{https://web.stanford.edu/~boyd/papers/admm/lasso/lasso.html}{https://web.stanford.edu/~boyd/papers/admm/lasso/lasso.html}}} (ADMM); see \cite{BPCPE,gabay,glomar}.
	\item[\bf (ii)] The {\em Accelerated Proximal Gradient Method\footnote{\href{https://github.com/bodono/apg}{https://github.com/bodono/apg}}} (APG); see \cite{Nesterov,nest}.
	\item[\bf (iii)] The {\em Fast Iterative Shrinkage-Thresholing
		Algorithm\footnote{\href{https://github.com/he9180/FISTA-lasso}{https://github.com/he9180/FISTA-lasso}}} (FISTA); see \cite{BeckTebou}.
	\item[\bf(iv)] The {\em Semismooth Newton Augmented Lagrangian
		Method\footnote{\href{https://www.polyu.edu.hk/ama/profile/dfsun/}{https://www.polyu.edu.hk/ama/profile/dfsun/}}} (SSNAL) developed in \cite{lsk}.
\end{itemize}
All the numerical experiments are conducted on a desktop with 10th Gen Intel(R) Core(TM) i5-10400 processor (6-Core, 12M Cache, 2.9GHz to 4.3GHz) and 16GB memory. All the codes are written in MATLAB 2016a. Our numerical experiments are conducted with the test instances $(A, b)$ in \eqref{Lasso} generated randomly following the Matlab commands
\begin{align*}
    A={\rm randn}(m,n);\quad b={\rm randn}(n,1).
\end{align*}
In order to run Algorithm~\ref{LSQP} for solving Lasso problems, the matrix $A^*A$ needs to be positive-definite due to Theorem~\ref{solveLasso}. Since $\text{\rm rank}(A^*A)=\text{\rm rank}A$, the matrix $A^*A$ is singular if $m<n$. Therefore, a necessary condition for the positive-definite of $A^*A$ is that $m\ge n$, and thus we only test datasets in which the matrix $A$ has more rows than columns, or its rows and columns are equal. Note that the modes with $m\ge n$ appear in practical applications; see, e.g., \cite{EHJT04} with applications to diabetes studies and \cite{BeckTebou} with applications to image processing. The GDNM code is publicly available on the website\footnote{\href{https://github.com/he9180/GDNM/}{https://github.com/he9180/GDNM/}}.}

\medskip
The initial points in all the experiments are set to be the zero vector. The following \textit{relative KKT residual} $\eta_k$ in
\eqref{KKT} suggested in \cite{lsk} is used to  measure the accuracy of an approximate optimal solution $x^k$ for \eqref{Lasso}:
\begin{equation}\label{KKT}
\eta_k := \frac{\|x^k - \text{\rm Prox}_{\mu\|\cdot\|_1}(x^k-A^*(Ax^k-b)) \|}{1+\|x^k\|+\|Ax^k-b\|}.
\end{equation}
We stop the algorithms in our experiments when either the condition $\eta_k <10^{-6}$ is satisfied, or they reach the maximum computation time of $6000$ seconds. For testing purpose, the regularization parameter $\mu$ in the Lasso problem \eqref{Lasso} is
chosen as $10^{-3}$ or as in \cite{lsk}, i.e.
\begin{equation}\label{para}
\mu= 10^{-3}\|A^*b\|_{\infty}, \quad \text{where }\; \|x\|_{\infty}:= \max\{|x_1|,|x_2|,\ldots, |x_n| \}, x= (x_1,x_2,\ldots,x_n).
\end{equation} 
The achieved numerical results are presented in Table \ref{table 1} and Table \ref{table 2}. In these tables, ``CPU time" stands for the time needed to achieve the prescribed accuracy of approximate solutions (the smaller the better). As we can see from the results presented in Tables \ref{table 1} and \ref{table 2}, GDNM is more efficient in the cases where $m >>n$ and $n$ is not large, which confirms the need of positive definiteness of the matrix $A^*A$ for the superlinear convergence. On the other hand, ADMM performs well in all tests while SSNAL is more efficient for large-scale datasets.
\vspace*{0.05in}
\begin{sidewaystable}[!htbp]
	\begin{minipage}{\textheight}\small
		\caption{$\mu=10^{-3}$ }\label{table 1}
		\centering
		\begin{tabular}{llllllllllllll} 
			\hline
			\multicolumn{2}{c}{Problem size} &  & \multicolumn{5}{c}{Number of iterations}               &  & \multicolumn{5}{c}{CPU time}                    \\ 
			\hline
			m    & n                         &  & SSNAL & GDNM & ADMM  & APG    & FISTA  &  & SSNAL   & GDNM    & ADMM   & APG     & FISTA    \\
			\hline
			1024 & 256                       &  & 4     & 4    & 9     & 42     & 133    &  & 0.13    & 0.09    & 0.02   & 0.04    & 0.25     \\
			1024 & 1024                      &  & 30    & 22   & 12192 & 172326 & 190461 &  & 10.68   & 1.23    & 35.71  & 485.51  & 2046.34  \\
			4096 & 256                       &  & 4     & 4    & 9     & 24     & 44     &  & 0.14    & 0.11    & 0.03   & 0.05    & 0.34     \\
			4096 & 4096                      &  & 1306  & 281  & 12571 & 114879 & 38912  &  & 4163.54 & 6000.00 & 775.59 & 6000.00 & 6000.00  \\
			\hline
		\end{tabular}
		
		\medskip
		
		\caption{$\mu=10^{-3}\left\|A^*b\right\|_{\infty}$ }\label{table 2}
		\begin{tabular}{llllllllllllll} 
			\hline
			\multicolumn{2}{c}{Problem size} &  & \multicolumn{5}{c}{Number of iterations}           &  & \multicolumn{5}{c}{CPU time}               \\ 
			\hline
			m    & n                         &  & SSNAL & GDNM & ADMM & APG  & FISTA &  & SSNAL  & GDNM   & ADMM  & APG   & FISTA    \\
			\hline
			1024 & 256                       &  & 4     & 5    & 123  & 45   & 133   &  & 0.62   & 0.11   & 0.03  & 0.03  & 0.24     \\
			1024 & 1024                      &  & 17    & 172  & 174  & 2638 & 26431 &  & 2.38   & 10.72  & 0.56  & 7.41  & 273.84   \\
			4096 & 256                       &  & 4     & 4    & 248  & 26   & 44    &  & 0.22   & 0.12   & 0.15  & 0.05  & 0.32     \\
			4096 & 4096                      &  & 18    & 355  & 343  & 1797 & 32412 &  & 149.57 & 668.32 & 23.68 & 95.60 & 5121.51  \\
			\hline
		\end{tabular}

	\end{minipage}
\end{sidewaystable}

\newpage 
\section{Concluding Remarks and Further Research}\label{sec:conclusion}

This paper proposes and develops new globally convergent algorithms
of the damped Newton type to solve some classes of nonsmooth
optimization problems addressing minimization of ${\cal C}^{1,1}$
objectives and problems of quadratic composite optimization with
extended-real-valued regularizers, which include nonsmooth problems
of constrained optimization. We verify well-posedness of the
proposed algorithms and their linear and superlinear convergence
under rather nonrestrictive assumptions. Our approach is based on
advanced machinery of second-order variational analysis and
generalized differentiation. The obtained results are applied to
some classes of optimization problems that arise in machine
learning, statistics, and related areas with the efficient
implementation to solving the well-recognized Lasso problems. The
numerical experiments conducted to solve {an important class of nonsmooth} Lasso problems by using the suggested algorithm are compared in
detail with the corresponding calculations by using some other
first-order and second-order algorithms.

Our future research includes efficient calculations of second-order
subdifferentials and proximal mappings used in this paper for
broader classes of convex and nonconvex problems with further
applications to practically important models of machine learning,
statistics, etc. We also intend to establish a global superlinear
convergence of our damped generalized Newton algorithms for problems
of quadratic composite optimization with extended-real-valued
regularizers without the positive-definiteness requirement on the
quadratic term.\vspace*{-0.1in}

\section*{Acknowledgments}\label{sec:acknowledgement}   The authors are grateful to two anonymous referees for their helpful remarks and comments that allowed us to significantly improve the original manuscript. Our thanks also go to Alexey Izmailov for his useful remarks on the
algorithm developed in Section~3 and to  Michal Ko\v cvara and Defeng Sun
for helpful discussions on numerical experiments to solve
Lasso problems.

\end{document}